\documentclass[10pt]{article}
\usepackage{lmodern}
\usepackage{amsmath}
\usepackage[T1]{fontenc}
\usepackage[utf8]{inputenc}
\usepackage{authblk}
\usepackage{amsfonts}
\usepackage{graphicx}
\usepackage{rotating}
\usepackage{amssymb}
\usepackage[english]{babel}
\usepackage{color}
\usepackage{amsthm}
\usepackage{graphicx}
\usepackage{tkz-euclide}
\tikzset{elegant/.style={smooth,thick,samples=50,cyan}}
\usepackage{color,tikz}
\usetikzlibrary{patterns}
\usetikzlibrary{decorations.pathreplacing,calc}
\usepackage{rotating}
\usepackage{mathrsfs}
\usepackage{makecell}
\usepackage{microtype}
\usepackage{enumerate}
\usepackage{mathscinet}
\usepackage{array}
\usepackage{multirow}
\usepackage{booktabs}%table
\usepackage{enumerate}
\usepackage[cal=boondoxo,bb=ams]{mathalfa}
\usepackage{hyperref}
\hypersetup{hidelinks}
\usepackage{titlesec}
\newtheorem{theorem}{Theorem}[section]
\newtheorem{prop}{Proposition}[section]

\newtheorem{coro}{Corollary}[section]
\newtheorem{remark}{Remark}[section]

\newcommand{\ml}{\mathcal}
\newcommand{\mb}{\mathbb}

\DeclareMathOperator{\intt}{int}
\DeclareMathOperator{\extt}{ext}
\DeclareMathOperator{\bdd}{bdd}

\title{Optimal large-time estimates and singular limits for thermoelastic plate equations with the Fourier law}
\author[1]{Wenhui Chen\thanks{Wenhui Chen (wenhui.chen.math@gmail.com)}}
\affil[1]{School of Mathematics and Information Science, Guangzhou University, 510006 Guangzhou, China}
\author[2]{Ryo Ikehata\thanks{Ryo Ikehata (ikehatar@hiroshima-u.ac.jp)}}
\affil[2]{Department of Mathematics, Division of Educational Sciences, Graduate School of Humanities and Social Sciences, Hiroshima University, 739-8524 Higashi-Hiroshima, Japan}

\date{}
\begin{document}

\maketitle
\begin{abstract}
	\medskip
In this paper, we study asymptotic behaviors for classical thermoelastic plate equations with the Fourier law of heat conduction in the whole space $\mathbb{R}^n$, where we introduce a reduction methodology basing on third-order (in time) differential equations and refined Fourier analysis.  We derive optimal growth estimates when $n\leqslant 3$, bounded estimates when $n=4$, and decay estimates when $n\geqslant 5$ for the vertical displacement in the $L^2$ norm. Particularly, the new critical dimension $n=4$ for distinguishing the decisive role between the plate model and the Fourier law of heat conduction is discovered.  Moreover, concerning the small thermal parameter in the temperature equation, we study the singular limit problem. We not only show global (in time) convergence of the vertical displacements between thermoelastic plates and  structurally damped plates, but also rigorously demonstrate a new second-order profile of the solution. Our methodology can settle several closely related problems in thermoelasticity.\\
	
	
	\noindent\textbf{AMS Classification (2020)} 35G40, 35B40, 35Q79 
\end{abstract}
\fontsize{12}{15}
\selectfont
%\tableofcontents
\section{Introduction}
It is widely known that thin plates theory arises in numerous applications of engineering, for example, raft foundation, road pavement, airport runway, etc. In order to describe thin plate motions under different practical circumstances, there are various mathematical models in the plate field, including Mindlin-Timoshenko models, von K\'arm\'an models, thermoelastic plate equations and viscoelastic plates. In the paper, we will study some asymptotic behaviors (in terms of large-time or small thermal parameter) for classical thermoelastic plate equations in the whole space $\mb{R}^n$. Before introducing our aims, we will sketch out some historical background of thermoelastic plates.
\subsection{Background of thermoelastic plate equations}\label{Sub-Sec-Background}
Let us consider a homogeneous, elastic and thermally isotropic plate subjecting to a temperature distribution. Associating with the second law of thermodynamics for irreversible process (relates the entropy to the elastic strains), the well-known thermoelastic plate equations with the Fourier law of heat conduction, namely,
\begin{align}\label{Eq-G-Linear-TEP}
	\begin{cases}
		u_{tt}+\Delta^2 u+\Delta\theta=0,\\
		\epsilon\theta_t-\Delta \theta-\Delta u_t=0,
	\end{cases}
\end{align}
has been established, where the scalar unknowns $u=u(t,x)$ and $\theta=\theta(t,x)$ denote, respectively, the vertical displacement and  the temperature (relative to some reference temperature). The constant $\epsilon>0$, which is the so-called \emph{thermal parameter}, consists of some parameters from the specific heat of the body, thermal conductivity as well as the mass density per unit volume. Note that we omitted other physical constants for the simplicity. Concerning heuristic derivations of the model \eqref{Eq-G-Linear-TEP}, we refer interested readers to \cite{Lagnese-Lions=1988,Lagnese=1989}.

In recent years, thermoelastic plate equations \eqref{Eq-G-Linear-TEP} in bounded or unbounded domains (even the whole space $\mb{R}^n$) have caught a lot of attentions (see \cite{Kim=1992,Shibata=1994,Munoz-Rivera-Racke=1995,Liu-Renardy=1995,Liu-Zheng=1997,Liu-Liu=1997,Lasiecka-Triggiani=1998,Liu-Yong=1998,Lebeau-Zuazua=1998,Perla-Zuazua=2003,Denk-Racke=2006,Naito-Shibata=2009,Said-Houari=2013,Lasiecka-Wilke=2013,Racke-Ueda=2016,Racke-Ueda=2017,Denk-Shibata=2017,Denk-Shibata=2019} and references therein). We stress that topics of thermoelastic plates are energetic not only in the communities of PDEs but also of controllability and dynamical systems perspectives. According to the theme of this work, we will just briefly introduce the progressive progress in the corresponding Cauchy problem to \eqref{Eq-G-Linear-TEP} with $\epsilon=1$, namely,
\begin{align}\label{Eq-C-Linear-TEP}
	\begin{cases}
		u_{tt}+\Delta^2 u+\Delta\theta=0,&x\in\mb{R}^n,\ t>0,\\
		\theta_t-\Delta \theta-\Delta u_t=0,&x\in\mb{R}^n,\ t>0,\\
		u(0,x)=u_0(x),\ u_t(0,x)=u_1(x),\ \theta(0,x)=\theta_0(x),&x\in\mb{R}^n.
	\end{cases}
\end{align}
By constructing a natural ansatz (motivated by its energy terms) of the coupled system \eqref{Eq-C-Linear-TEP} as follows:
\begin{align}\label{old-unknown}
U^{\mathrm{N}}(t,x):=\big(u_t(t,x),\Delta u(t,x),\theta(t,x)\big),	
\end{align}
 the author of \cite{Said-Houari=2013} obtained some decay estimates of $U^{\mathrm{N}}(t,\cdot)$ in the $L^{2}$ or $L^{\infty}$ norms with weighted $L^1$ datum, whose proof is based on energy method as well as Fourier splitting techniques. Later, associated with energy estimates in the Fourier space, the authors of \cite{Racke-Ueda=2016} derived sharp decay estimates for $U^{\mathrm{N}}(t,\cdot)$ in the $\dot{H}^k$ norm carrying $k\geqslant 0$, where initial datum are assumed in energy spaces with additionally $L^1$ regularity. In order to verify the optimality of the next decay estimate (it also was shown in \cite[Theorem 3.6]{Munoz-Rivera-Racke=1995} and \cite[Theorem 2.1 with $\gamma=0$]{Said-Houari=2013}):
 \begin{align*}
 	\|U^{\mathrm{N}}(t,\cdot)\|_{\dot{H}^k}\lesssim (1+t)^{-\frac{n}{4}-\frac{k}{2}}\|U^{\mathrm{N}}(0,\cdot)\|_{ L^1}+\mathrm{e}^{-ct}\|U^{\mathrm{N}}(0,\cdot)\|_{\dot{H}^k},
 \end{align*}
the characteristic roots of \eqref{Eq-C-Linear-TEP} in the Fourier space have been calculated explicitly. Additionally, from their works, there exists a unique solution to \eqref{Eq-C-Linear-TEP} satisfying
\begin{align*}
u\in\ml{C}([0,\infty),\dot{H}^2),\ \ u_t\in\ml{C}([0,\infty),L^2),\ \ \theta\in\ml{C}([0,\infty),L^2),
\end{align*}
for $u_0\in \dot{H}^2$ and $u_1,\theta_0\in L^2$ originated from the vector unknown \eqref{old-unknown}.  Subsequently, a quasilinear Cauchy problem was investigated by \cite{Racke-Ueda=2017}.  Under this situation, a natural question emerges as follows:
\begin{center}
\emph{How does the Fourier law of heat conduction influence on the plate model for large-time?}
\end{center}
From the viewpoint of the unknown \eqref{old-unknown}, a series of works \cite{Munoz-Rivera-Racke=1995,Said-Houari=2013,Racke-Ueda=2016,Racke-Ueda=2017} claimed that the Fourier law of heat conduction plays a decisive role. Nevertheless, if one considers a corresponding nonlinear problem of \eqref{Eq-C-Linear-TEP} with the nonlinearity $\ml{N}[u]$ of the solution itself $u=u(t,x)$, an important step is to understand the above-mentioned question with respect to $u$ rather than $U^{\mathrm{N}}$. Concerning large-time behaviors of the solution itself, a complete classification is still unknown, but we will state the new classification in the $L^2$ framework by the critical dimension $n=4$ in this paper. In particular, the large-time behaviors of $u(t,\cdot)$ to the coupled system \eqref{Eq-C-Linear-TEP} are fully determined by the plate model in physical dimensions $n=1,2,3$ as a new phenomenon.

To the best of authors' knowledge, asymptotic behaviors (e.g. estimates in the $L^2$ norm and asymptotic profiles) of the vertical displacement $u(t,\cdot)$, so far did not be investigated deeply. Indeed, these studies are not generalizations of the previous works \cite{Munoz-Rivera-Racke=1995,Said-Houari=2013,Racke-Ueda=2016} due to some hidden influences of oscillations from the plate model . Let us explain this challenge briefly by an example.
\begin{remark}
 If one uses the sharp quantitative estimate in \cite[Equation (2.3)]{Racke-Ueda=2016} that is
\begin{align*}
	|\xi|^2|\hat{u}(t,\xi)|\leqslant\frac{13}{3}\mathrm{e}^{-\frac{|\xi|^2t}{52}}|\widehat{U}^{\mathrm{N}}(0,\xi)|,
\end{align*}
by simply dividing $|\xi|^2$ in the previous line, the next unbounded estimates appear:
\begin{align*}
	\|u(t,\cdot)\|_{L^2}&\lesssim\left\||\xi|^{-2}\mathrm{e}^{-c|\xi|^2t}\right\|_{L^2(|\xi|\leqslant 1)}\|\widehat{U}^{\mathrm{N}}(0,\xi)\|_{L^{\infty}(|\xi|\leqslant 1)}+\mathrm{e}^{-ct}\|\widehat{U}^{\mathrm{N}}(0,\xi)\|_{\dot{H}^{-2}(|\xi|\geqslant 1)}\\
	&\lesssim\left(\int_0^1r^{n-5}\mathrm{e}^{-2cr^2t}\mathrm{d}r\right)^{1/2}\|U^{\mathrm{N}}(0,\cdot)\|_{L^1}+\mathrm{e}^{-ct}\|U^{\mathrm{N}}(0,\cdot)\|_{L^2}
\end{align*}
due to some singularities of $r^{n-5}$ for $n=1,\dots,4$ as $r\downarrow0$.
\end{remark}
\noindent Furthermore, another difficulty is the rigorous justification of optimal estimates in the sense of same upper and lower bound estimates for large-time. Especially, because some oscillating effects in thermoelastic plate equations \eqref{Eq-C-Linear-TEP}, one may expect growth estimates rather than decay estimates for lower dimensions as $t\gg1$.

Finally, we underline that the difficulties mentioned in preceding part of the paragraph not only occur in thermoelastic plates, but also appear in other coupled systems, e.g. thermoelasticity (with the Cattaneo law, the Pipkin-Gurtin law, or the Jeffreys type equation), porous elasticity, Timoshenko beams and Bresse beams because applications of energy methods and energy-type ansatz. Over the last couple of decades, another effective approach to treat linear coupled systems, which is the so-called \emph{diagonalization procedure}, had been  deeply developed (see, for example, \cite{Reissig-Wang=2005,Yang-Wang=2006,Jachmann-Reissig=2009,Liu-Reissig=2014,Reissig=2016,Liu-Shi=2022}). Some asymptotic representations of a micro-energy can be formulated by diagonalization procedure, but sharp asymptotic behaviors of the solution itself are still open to the best of authors' knowledge. To overcome these difficulties,  with a deep root from the monograph \cite{Jiang-Racke=2000}, we will propose a \emph{reduction methodology}, that is a reduction to higher-order (in time) PDEs associated with refined time-frequency analysis, to deal with thermoelastic plate equations \eqref{Eq-C-Linear-TEP}. This approach can recover some oscillations to compensate strong singularities in lower dimensions (see Remark \ref{Rem-3.3}) and enlighten the construction of asymptotic profiles. More detailed explanation will be addressed afterwards. We believe our methodology can be widely applied in other hyperbolic-parabolic coupled systems and hyperbolic-hyperbolic coupled systems, e.g. thermoelastic system of type II, of type III, or with second sound, dissipative Timoshenko system, Bresse systems.

\subsection{Main purposes of this paper}
In the present manuscript, we consider the following Cauchy problem for classical thermoelastic plate equations with the Fourier law of heat conduction:
\begin{align}\label{Eq-Linear-TEP}
	\begin{cases}
		u_{tt}+\Delta^2 u+\Delta\theta=0,&x\in\mb{R}^n,\ t>0,\\
		\epsilon\theta_t-\Delta \theta-\Delta u_t=0,&x\in\mb{R}^n,\ t>0,\\
		u(0,x)=u_0(x),\ u_t(0,x)=u_1(x),\ \theta(0,x)=\theta_0(x),&x\in\mb{R}^n,
	\end{cases}
\end{align}
with a positive constant $\epsilon$ as we introduced in Subsection \ref{Sub-Sec-Background}. Our first purpose  is to study large-time asymptotic behaviors of solutions to \eqref{Eq-Linear-TEP} carrying $\epsilon=1$ without loss of generality. In Section \ref{Sec-Large-time}, by applying the reduction procedure for the coupled system, the core consideration immediately becomes the third-order (in time) evolution equation \eqref{Eq-third-order}. Then, with the aid of refined Fourier analysis, we derive optimal estimates of solutions in the $L^2$ norm. Particularly among them, the next optimal estimates hold for $t\gg1$:
\begin{align*}
\|u(t,\cdot)\|_{L^2}\simeq
	t^{1-\frac{n}{4}}\ \ \mbox{for any}\ \ n\geqslant 1,
\end{align*}
under $|P_{u_1}|\neq 0$ and some regular hypotheses for initial datum.  It implies infinite time $L^2$-blowup of the vertical displacement when $n=1,2,3$ with optimal polynomial growth rates (the decisive part is the plate model); neither growth nor decay estimates of the vertical displacement when $n=4$ (exquisite interplay between plates and the Fourier law); optimal decay estimates with the factors from plates and the Fourier law when $n\geqslant 5$.  One may see Remark \ref{Rem-new} for the  in-depth explanation.  As a by-product, we determine asymptotic profiles of solutions by diffusive-plates and the heat model. 

 For another, let us formally take into account the limit case with the vanishing thermal parameter $\epsilon=0$ in thermoelastic plate equations \eqref{Eq-Linear-TEP}, which turns out to be structurally damped plate equation (see \cite{Pham-Kainane-Reissig=2015,Dabbicco-Ebert=2017,Dao-Reissig=2019} and references therein) whose initial datum are inherited from \eqref{Eq-Linear-TEP} as follows:  
\begin{align}\label{Eq-u,theta,0}
	\begin{cases}
		u_{tt}^0+\Delta^2 u^0-\Delta u_t^0=0,&x\in\mb{R}^n,\ t>0,\\
		u^0(0,x)=u_0(x),\ u_t^0(0,x)=u_1(x),&x\in\mb{R}^n.
	\end{cases}
\end{align}
Formally, \eqref{Eq-u,theta,0} seems to be the limit model for the coupled system \eqref{Eq-Linear-TEP} with  $\epsilon=0$ whereas the last data $\theta(0,x)=\theta_0(x)$ will lose. For this reason, our second purpose is to investigate the small-parameter asymptotic behaviors of solutions to \eqref{Eq-Linear-TEP}. In Section \ref{Sec-Singular-limit}, we rigorously justify singular limits (limiting procedures) from thermoelastic plate equations \eqref{Eq-Linear-TEP} to structurally damped plate equation \eqref{Eq-u,theta,0} as the thermal parameter tends to zero. For instance, when $n\geqslant 3$, it holds 
\begin{center}
 $u\to u^0$ in $L^{\infty}([0,\infty),L^2(\mb{R}^n))$ as $\epsilon\downarrow 0$ with the convergence rate $\epsilon$.
\end{center}
 Furthermore, by applying multi-scale analysis and energy methods in the Fourier space, we derive a formal expansion of solutions with respect to $\epsilon$, and rigorously justify the correction for the second-order asymptotic expansion. For example,  when $n\geqslant 3$, it holds  
 \begin{center}
 	$u\to u^0+\epsilon u^{I,1}$ in $L^{\infty}([0,\infty),L^2(\mb{R}^n))$ as $\epsilon\downarrow 0$  with the  convergence rate $\epsilon^2$,
 \end{center}
 where $u^{I,1}=u^{I,1}(t,x)$ is the solution to inhomogeneous structural damped plates with the source term $\Delta u_t^0-\Delta^2u^0$ (see precisely in \eqref{uI1} later).

This paper is organized as follows: Section \ref{Sec-Main} will show main results of this work, whose proofs will be given in Section \ref{Sec-Large-time} and Section \ref{Sec-Singular-limit}, respectively. Eventually, some concluding remarks in Section \ref{Sec-Concluding} about some applications of our methodology complete the manuscript.

\subsection{Notations}\label{Sub-Notation}
Let us introduce some notations which will be used all over this work. We define the following zones of the Fourier space:
\begin{align*}
	\ml{Z}_{\intt}(\varepsilon_0):=\{|\xi|\leqslant\varepsilon_0\ll1\},\ \ 
	\ml{Z}_{\bdd}(\varepsilon_0,N_0):=\{\varepsilon_0\leqslant |\xi|\leqslant N_0\},\ \ 
	\ml{Z}_{\extt}(N_0):=\{|\xi|\geqslant N_0\gg1\},
\end{align*}
in which, the cut-off functions $\chi_{\intt}(\xi),\chi_{\bdd}(\xi),\chi_{\extt}(\xi)\in \mathcal{C}^{\infty}$ having their supports in the corresponding zones $\ml{Z}_{\intt}(\varepsilon_0)$, $\ml{Z}_{\bdd}(\varepsilon_0/2,2N_0)$ and $\ml{Z}_{\extt}(N_0)$, respectively, such that
\begin{align*}
	\chi_{\bdd}(\xi)=1-\chi_{\intt}(\xi)-\chi_{\extt}(\xi)\ \ \mbox{for all}\ \ \xi \in \mb{R}^n.
\end{align*}
The symbols of pseudo-differential operators $|D|^s$ and $\langle D\rangle^s $ are denoted by $|\xi|^s$ and $\langle \xi\rangle^s$ with $s\in\mb{R}$, respectively, where $\langle \xi\rangle:=\sqrt{1+|\xi|^2}$ is  the Japanese bracket.

The symbol $f\lesssim g$ means: there exists a positive constant $C$ fulfilling $f\leqslant Cg$, which may be varied in different lines, analogously, for $f\gtrsim g$. Additionally, the relation $f\simeq g$ holds if and only if $f\lesssim g$ and $f\gtrsim g$ are retained meanwhile.  

To end the introduction, let us recall the weighted $L^1$ space
\begin{align*}
	L^{1,1}:=\left\{f\in L^1 \ \big|\ \|f\|_{L^{1,1}}:=\int_{\mb{R}^n}(1+|x|)|f(x)|\mathrm{d}x<\infty \right\}
\end{align*}
so that $\|f\|_{L^1}\leqslant\|f\|_{L^{1,1}}$ notably. The mean of a summable function $f$ is denoted by $P_f:=\int_{\mb{R}^n}f(x)\mathrm{d}x$. The Sobolev spaces of negative order are defined by (see, for example, \cite{Runst-Si-book})
\begin{align*}
	\dot{H}^{-s}:=\left\{f\in\ml{Z}' \ \big|\ \|f\|_{\dot{H}^{-s}}:=\|(-\Delta)^{-\frac{s}{2}}f\|_{L^2}<\infty \right\},
\end{align*}
where $\ml{Z}'$ is the factor space $\ml{S}'\backslash \ml{P}$ with $\ml{P}$ denoting the space of all polynomials. We mark the usual $L^2$ space with the localization $\chi_{\intt}(D)$ or $\chi_{\intt}(\xi)$ by $L^2_{\chi}$.

\section{Main results}\label{Sec-Main}
\subsection{Asymptotic behaviors of solutions for large-time}
To express the roots of the characteristic equation (see the statements after \eqref{CC-01}) for the Cauchy problem \eqref{Eq-Linear-TEP}, let us introduce two positive parameters
\begin{align}\label{alpha+-}
	\alpha_{\pm}:=\sqrt[3]{\frac{1}{2}\left(3\sqrt{69}+11\right)}\pm\sqrt[3]{\frac{1}{2}\left(3\sqrt{69}-11\right)}.
\end{align}
We firstly formulate optimal estimates of the solution $(u,\theta)$ to thermoelastic plate equations.
\begin{theorem}\label{Thm-Linear} 
	Let us consider the Cauchy problem \eqref{Eq-Linear-TEP} with $\epsilon=1$ and initial datum $u_0\in H^2\cap L^1$, $u_1,\theta_0\in L^2\cap L^{1}$ for any $n\geqslant 1$. Then, the solution $(u,\theta)$ satisfies the following estimates for $t\gg1$:
	\begin{align*}
		\|u(t,\cdot)\|_{L^2}&\lesssim t^{-\frac{n}{4}}\|u_0\|_{L^2\cap L^1}+t^{1-\frac{n}{4}}\|(u_1,\theta_0)\|_{(L^2\cap L^1)^2},\\
		\|\theta(t,\cdot)\|_{L^2}&\lesssim t^{-1-\frac{n}{4}}\|u_0\|_{\dot{H}^2\cap L^1}+t^{-\frac{n}{4}}\|(u_1,\theta_0)\|_{(L^2\cap L^1)^2}.
	\end{align*}
Assuming  $u_1,\theta_0\in L^{1,1}$ additionally, then the solution $(u,\theta)$ satisfies the following optimal estimates for $t\gg1$:
\begin{align*}
		t^{1-\frac{n}{4}}|P_{u_1}|&\lesssim \|u(t,\cdot)\|_{L^2}\lesssim t^{1-\frac{n}{4}}\|(u_0,u_1,\theta_0)\|_{(L^2\cap L^1)\times(L^2\cap L^{1})^2} ,\\
	 t^{-\frac{n}{4}}|P_{\theta_0}|&\lesssim \|\theta(t,\cdot)\|_{L^2}\lesssim t^{-\frac{n}{4}}\|(u_0,u_1,\theta_0)\|_{(\dot{H}^2\cap L^1)\times (L^2\cap L^{1})^2}.
\end{align*}
\end{theorem}
\begin{remark}\label{Rem-new}
Our first innovation is the new classification with respect to the dimension $n$ of interplay between the plate model and the Fourier law of heat conduction. Recalling the large-time behavior of the pure plate model
\begin{align}\label{Eq-Plates}
\begin{cases}
v_{tt}+\Delta^2v=0,&x\in\mb{R}^n,\ t>0,\\
v(0,x)=v_0(x),\ v_t(0,x)=v_1(x),&x\in\mb{R}^n,
\end{cases}
\end{align}
under some $L^2$ assumptions with additionally  $L^1$ regularity (even $L^{1,1}$ regularity) on initial datum and $|P_{u_1}|\neq0$, the recent work \cite{Ikehata=2021-Dec} got the optimal growth estimates
\begin{align*}
\|v(t,\cdot)\|_{L^2}^2\simeq\begin{cases}
t^{2-\frac{n}{2}}&\mbox{if}\ \ n\leqslant 3,\\
\log t&\mbox{if}\ \  n=4,
\end{cases}
\end{align*}
for $t\gg1$. Reviewing Theorem \ref{Thm-Linear}, we actually obtained the optimal estimates of thermoelastic plate equations \eqref{Eq-Linear-TEP} with $\epsilon=1$ by the following one:
\begin{align*}
\|u(t,\cdot)\|_{L^2}^2\simeq t^{2-\frac{n}{2}}=\begin{cases}
t^{2-\frac{n}{2}}&\mbox{if}\ \ n\leqslant 3,\\
1&\mbox{if}\ \ n=4,\\
t^{-\frac{n-4}{2}}&\mbox{if}\ \ n\geqslant 5,
\end{cases}
\end{align*}
as $t\gg1$.  Then, we may realize that the growth rates for the pure plate model \eqref{Eq-Plates} and thermoelastic plate equations \eqref{Eq-C-Linear-TEP} if $n\leqslant 3$ are the same;  the growth rate $\log t$ in the pure plate equation \eqref{Eq-Plates} if $n=4$ can be improved by the bounded one due to the Fourier law of heat conduction from thermoelastic plates \eqref{Eq-C-Linear-TEP};  the decay rates $t^{2-\frac{n}{2}}$ for thermoelastic plate equations if $n\geqslant 5$  is contributed by two portions: the factor $t^2$ comes from the bi-Laplacian operator in the plate model \eqref{Eq-Plates} and another factor $t^{-\frac{n}{2}}$ originates from the heat conduction since the Gaussian kernel. The heat conduction generates decay properties. In conclusion, by the summary in Table \ref{tab:table1}, we may answer the question proposed in the introduction.
\renewcommand\arraystretch{1.4}
\begin{table}[h!]
	\begin{center}
		\caption{Influence from the plate model and the Fourier law of heat conduction}
		\label{tab:table1}
		\begin{tabular}{cccc} 
			\toprule
			\multirow{2}*{Dimensions} & $n\leqslant 3$ & $n=4$ & $n\geqslant 5$\\
			 & (Lower-dimensions) & (Critical-dimension) & (Higher-dimensions)  \\
			\midrule
			Pure plates property & $t^{2-\frac{n}{2}}$ & $\log t$ & -- \\
			Heat's property (the Fourier law) & $t^{-\frac{n}{2}}$& $t^{-2}$& $t^{-\frac{n}{2}}$\\  
			Thermoelastic plate's property & $t^{2-\frac{n}{2}}$ & $1$ & $t^{-\frac{n-4}{2}}$\\
			\hline 
			Crucial influence & Pure plates & Pure plates + the Fourier law & the Fourier law\\
			\bottomrule
			\multicolumn{4}{l}{\emph{$*$The terminology ``property'' specializes the time-dependent coefficient in the $L^2$ estimates of the solution.}} 
		\end{tabular}
	\end{center}
\end{table}

\noindent It is worth noting that all large-time properties in Table \ref{tab:table1} are optimal in the sense of same behaviors for upper and lower bounds of the vertical displacement in the $L^2$ norm. In particular, concerning physical dimensions $n=1,2,3$, the large-time properties of thermoelastic plate equations \eqref{Eq-Linear-TEP} are not influenced by the heat conduction.
\end{remark}
\begin{remark}
The optimal estimates for $(u,\theta)$ not only imply some underlying phenomena of linear thermoelastic plates, but also prepare crucial estimates to study global (in time) existence of solutions to nonlinear models. For examples, \cite{Bezerra-Carbone-Nascimento=2018,Bezerra-Carbone-Nascimento=2019} considered \eqref{Eq-Linear-TEP} with $\ml{N}[u]$ on the right-hand side. Motivated by Theorem \ref{Thm-Linear}, by constructing a time-weighted Sobolev space
\begin{align*}
\ml{Y}_s(T)&:=\sup\limits_{t\in[0,T]}\big((1+t)^{-1+\frac{n}{4}}\|u(t,\cdot)\|_{L^2}+(1+t)^{\frac{n}{4}}\|\theta(t,\cdot)\|_{L^2}\\
&\qquad\qquad\ \ +(1+t)^{\frac{s}{2}+\frac{n}{4}}\left(\|u(t,\cdot)\|_{\dot{H}^{2+s}}+\|\theta(t,\cdot)\|_{\dot{H}^s}\right)\big),
\end{align*}
we may prove global (in time) existence results for the Cauchy problem of thermoelastic plates with some nonlinear functions $\ml{N}[u]$ (e.g. the power-type nonlinearities $|u|^p$ or $||D|^au|^p$ with $0<a<2$) even $\ml{N}[u,\theta]$, whose philosophy is standard and similar to the one in \cite{Palmieri-Reissig=2018}.
\end{remark}

As a by-product, we also can describe asymptotic profiles for large-time. Before stating the corollary, let us introduce two functions
\begin{align}
	J_0(t,x)&:=\frac{1}{(a_0-a_1)^2+a_2^2}\ml{F}^{-1}_{\xi\to x}\left(\frac{1}{|\xi|^2}\left(\mathrm{e}^{-a_0|\xi|^2t}-\cos(a_2|\xi|^2t)\mathrm{e}^{-a_1|\xi|^2t}\right)\right),\label{M0}\\
	J_1(t,x)&:=\frac{1}{(a_0-a_1)^2+a_2^2}\ml{F}^{-1}_{\xi\to x}\left(\frac{\sin(a_2|\xi|^2t)}{a_2|\xi|^2}\mathrm{e}^{-a_1|\xi|^2t}\right),\label{M1}
\end{align} 
and the auxiliary functions with $j=0,1$ serving for the temperature variable as follows:
\begin{align*}
	J_{2+j}(t,x):=\ml{F}^{-1}_{\xi\to x}\left((|\xi|^{-2}\partial_t^2+|\xi|^2)\widehat{J}_j(t,|\xi|)\right),
\end{align*}
where $a_0:=(1+\alpha_-)/3$, $a_1:=(2-\alpha_-)/6$ and $a_2=\sqrt{3}\alpha_+/6$ are positive constants.
\begin{coro}\label{Coro-Linear}
Under the same assumption on initial datum as those in Theorem \ref{Thm-Linear}, the following refined estimates hold for $t\gg1$:
\begin{align*}
	\|u(t,\cdot)-J_0(t,\cdot)P_{\Psi_0}-J_1(t,\cdot)P_{\Psi_1}\|_{L^2_{\chi}}&\lesssim t^{\frac{1}{2}-\frac{n}{4}}\|(u_0,u_1,\theta_0)\|_{L^1\times(L^{1,1})^2},\\
	\|\theta(t,\cdot)-J_2(t,\cdot)P_{\Psi_0}-J_3(t,\cdot)P_{\Psi_1}\|_{L^2_{\chi}}&\lesssim t^{-\frac{1}{2}-\frac{n}{4}}\|(u_0,u_1,\theta_0)\|_{L^1\times(L^{1,1})^2},
\end{align*}
with the combined datum 
\begin{align*}
\Psi_0(x):=2a_1u_1(x)+\theta_0(x)\ \ \mbox{and} \ \ \Psi_1(x):=(a_0^2+a_2^2-a_1^2)u_1(x)+(a_0-a_1)\theta_0(x).
\end{align*}
\end{coro}
\begin{remark}
	The refined estimates in Corollary \ref{Coro-Linear} tell us that the solution $(u,\theta)^{\mathrm{T}}$ behaves like
	\begin{align*}
		\begin{pmatrix}
			J_0(t,x) & J_1(t,x)\\
			\big(|D|^{-2}\partial_t^2+|D|^2\big)J_0(t,x) & \big(|D|^{-2}\partial_t^2+|D|^2\big)J_1(t,x)\\
		\end{pmatrix}
	\begin{pmatrix}
			P_{\Psi_0} \\
		P_{\Psi_1} \\
	\end{pmatrix},
	\end{align*} as $t \to \infty$ in the $L^2$ framework, which gives the first-order profile. These dominant functions $J_0(t,x)$ and $J_1(t,x)$ can be understood by a linear combination of the diffusion-plates 
	\begin{align*}
		\ml{F}^{-1}_{\xi\to x}\left(
		\cos(a_2|\xi|^2t)
		\mathrm{e}^{-a_1|\xi|^2t}
		\right),\ \ 		\ml{F}^{-1}_{\xi\to x}\left(
		\frac{\sin(a_2|\xi|^2t)}{a_2|\xi|^2}
		\mathrm{e}^{-a_1|\xi|^2t}
		\right) 
	\end{align*}
	and the Gaussian kernel
	\begin{align*}
		\ml{F}^{-1}_{\xi\to x}\left(
		\mathrm{e}^{-a_0 |\xi|^2t}
		\right)
		=  
		\frac{1}{(4 \pi a_0 t)^{n/2}}
		\mathrm{e}^{-\frac{|x|^{2}}{4 a_0 t}}
	\end{align*}
	with the singularity $|\xi|^{-2}$ near $|\xi|=0$.
\end{remark}
\begin{remark} We may ensure an improvement $t^{-\frac{1}{2}}$ as $t\gg1$ by subtracting the corresponding profiles of $(u(t,\cdot),\theta(t,\cdot))$ in the $L^2$ norm for all dimensions. The effect is similar to generalized diffusion phenomena (see, for instance, \cite{Radu-Todorova-Yodanov=2016}).
\end{remark}

\subsection{Asymptotic behaviors of the solution for the small thermal parameter}
We state global (in time) convergence results between $u^{\epsilon}=u^{\epsilon}(t,x)$ of thermoelastic plate equations \eqref{Eq-Linear-TEP} and $u^0=u^0(t,x)$ of structurally damped plate equation \eqref{Eq-u,theta,0} in the $L^2$ framework.
\begin{theorem}\label{Thm-singu-lim}
	Let us consider thermoelastic plate equations \eqref{Eq-Linear-TEP} with the small thermal parameter $0<\epsilon\ll 1$, and structurally damped plate equation \eqref{Eq-u,theta,0}.
	\begin{itemize}
		\item Let $u_0\in H^2\cap L^1$, $u_1\in L^2\cap L^1$ and $u_1+\theta_0\in L^2$. Then, the convergence of energy terms holds:
		\begin{align*}
			&\sup\limits_{t>0}\left(\,\left\|u_t^{\epsilon}(t,\cdot)-u_t^{0}(t,\cdot)\right\|_{L^2}+\left\|\Delta u^{\epsilon}(t,\cdot)-\Delta u^0(t,\cdot)\right\|_{L^2}\right)\\
			&\qquad\leqslant C\epsilon \left(\|u_1+\theta_0\|_{L^2}+\|u_0\|_{H^2\cap L^1}+\|u_1\|_{L^2\cap L^1}\right)
		\end{align*}
	for any $n\geqslant 1$, where $C>0$ is independent of $t$ and $\epsilon$.
	\item Let $u_0\in L^2\cap L^1$, $u_1\in L^2\cap L^{1,1}$ with $|P_{u_1}|=0$ and $u_1+\theta_0\in \dot{H}^{-2}$. Then, the convergence of the vertical displacement holds:
	\begin{align*}
		\sup\limits_{t>0}\left\|u^{\epsilon}(t,\cdot)-u^0(t,\cdot)\right\|_{L^2}\leqslant C \epsilon\left(\|u_1+\theta_0\|_{\dot{H}^{-2}}+\|u_0\|_{L^2\cap L^1}+\|u_1\|_{L^2\cap L^{1,1}}\right)
	\end{align*}
for any $n\geqslant 3$, where $C>0$ is independent of $t$ and $\epsilon$.
	\end{itemize}
\end{theorem}
\begin{remark}
	The restriction $u_1\in L^{1,1}$ with $|P_{u_1}|=0$ is to guarantee global (in time) convergence result for $n\geqslant 3$. Instead of weighted $L^1$ hypothesis in our theorem, by assuming $u_1\in L^1$, we are able to get the convergence result for $n\geqslant 5$ only since
	\begin{align*}
	\int_0^t\left\|\chi_{\intt}(\xi)|\xi|^{-1}\mathrm{e}^{-c|\xi|^2\tau}\right\|_{L^2}^2\mathrm{d}\tau\lesssim\int_0^{\infty}(1+\tau)^{1-\frac{n}{2}}\mathrm{d}\tau<\infty,
	\end{align*}
when $1-n/2<-1$ that is $n>4$. If one considers a stronger assumption $u_1\in\dot{H}^{-2}$, we can arrive at a global (in time) convergence for any $n\geqslant 1$.
\end{remark}
\begin{remark}
The last theorem shows that
\begin{align*}
	u^{\epsilon}\to u^0\ \ &\mbox{in}\ \ L^{\infty}\big([0,\infty),H^2(\mb{R}^n)\big)\ \ \mbox{for}\ \ n\geqslant 3,\\
		u_t^{\epsilon}\to u_t^0,\ \ \Delta u^{\epsilon}\to\Delta u^0\ \ &\mbox{in}\ \ L^{\infty}\big([0,\infty),L^2(\mb{R}^n)\big)\,\ \ \mbox{for}\ \ n\geqslant 1,
\end{align*}
as $\epsilon\downarrow 0$ with the rate of convergence $\epsilon$. It implies a new relation for the thin plates with different dissipative mechanisms: the thermal damping versus the structural damping. By using the similar method to our proof additionally with the Hausdorff-Young inequality in the phase space, we expect some convergence results still hold in $L^{\infty}([0,\infty),H^{2,p}(\mb{R}^n))$ with $2\leqslant p<\infty$, even in $L^{\infty}([0,\infty)\times\mb{R}^n)$.
\end{remark}

Before showing  a faster convergence rate, let us introduce a new function
\begin{align}\label{uI1}
u^{I,1}(t,x)=-\frac{2}{\sqrt{3}}\int_0^t\sin\left(\frac{\sqrt{3}}{2}|D|^2(t-\tau)\right)\mathrm{e}^{-\frac{1}{2}|D|^2(t-\tau)}\left(u_t^0(\tau,x)-\Delta u^0(\tau,x)\right)\mathrm{d}\tau
\end{align}
as a second-order profile for the small $\epsilon>0$. Actually, $u^{I,1}(t,x)$ solves structurally damped plates with the source term $\Delta u^0_t(t,x)-\Delta^2u^0(t,x)$ on the right-hand side and vanishing initial conditions.
\begin{theorem}\label{Thm-Second}
		Let us consider thermoelastic plate equation \eqref{Eq-Linear-TEP} with $\theta_0(x)\equiv-u_1(x)$ as well as the small thermal parameter $0<\epsilon\ll 1$, and structurally damped plate equation \eqref{Eq-u,theta,0}.
	\begin{itemize}
		\item Let $u_0\in H^2\cap L^1$ and $u_1\in L^2\cap L^1$. Then, the further convergence of energy terms holds:
		\begin{align*}
			&\sup\limits_{t>0}\left(\,\left\|u_t^{\epsilon}(t,\cdot)-u_t^{0}(t,\cdot)-\epsilon u_t^{I,1}(t,\cdot)\right\|_{L^2}+\left\|\Delta u^{\epsilon}(t,\cdot)-\Delta u^0(t,\cdot)-\epsilon \Delta u^{I,1}(t,\cdot)\right\|_{L^2}\right)\\
			&\qquad\leqslant C\epsilon^2 \left(\|u_0\|_{H^2\cap L^1}+\|u_1\|_{L^2\cap L^1}\right)
		\end{align*}
		for any $n\geqslant 1$, where $C>0$ is independent of $t$ and $\epsilon$.
		\item Let $u_0\in L^2\cap L^1$ and $u_1\in L^2\cap L^{1,1}$ with $|P_{u_1}|=0$. Then, the further convergence of the vertical displacement holds:
		\begin{align*}
			\sup\limits_{t>0}\left\|u^{\epsilon}(t,\cdot)-u^0(t,\cdot)-\epsilon u^{I,1}(t,\cdot)\right\|_{L^2}\leqslant C \epsilon^2\left(\|u_0\|_{L^2\cap L^1}+\|u_1\|_{L^2\cap L^{1,1}}\right)
		\end{align*}
		for any $n\geqslant 3$, where $C>0$ is independent of $t$ and $\epsilon$.
	\end{itemize}
\end{theorem}
\begin{remark}
	The last theorem shows that
	\begin{align*}
		u^{\epsilon}\to u^0+\epsilon u^{I,1}\ \ &\mbox{in}\ \ L^{\infty}\big([0,\infty),H^{2}(\mb{R}^n)\big)\ \ \mbox{for}\ \ n\geqslant 3,\\
		u_t^{\epsilon}\to u_t^0+\epsilon u_t^{I,1},\ \ \Delta u^{\epsilon}\to\Delta u^0+\epsilon\Delta  u^{I,1}\ \ &\mbox{in}\ \ L^{\infty}\big([0,\infty),L^2(\mb{R}^n)\big)\, \ \ \mbox{for}\ \ n\geqslant 1,
	\end{align*}
	as $\epsilon\downarrow 0$ with the rate of convergence $\epsilon^2$. Comparing Theorem \ref{Thm-singu-lim} with Theorem \ref{Thm-Second}, by subtracting the additional terms in the $L^2$ norm, the rate of convergence has been improved by a factor $\epsilon$.
\end{remark}

\section{Large-time asymptotic behaviors of solutions}\label{Sec-Large-time}
This section will contribute to studies of large-time asymptotic profiles for linear thermoelastic plate equations \eqref{Eq-Linear-TEP}, namely, the proofs of Theorem \ref{Thm-Linear} and Corollary \ref{Coro-Linear}. To begin with, let us act the heat operator $\partial_t-\Delta$ on \eqref{Eq-Linear-TEP}$_1$ and plug \eqref{Eq-Linear-TEP}$_2$ with $\epsilon=1$ into the resultant, leading to
\begin{align*}
0&=(\partial_t-\Delta)(u_{tt}+\Delta^2 u)+\Delta(\partial_t-\Delta)\theta\\
&=u_{ttt}-\Delta u_{tt}+2\Delta^2 u_t-\Delta^3 u,
\end{align*}
with initial data $u_{tt}(0,x)=-\Delta^2 u_0(x)-\Delta\theta_0(x)$.  For this reason, we investigate the following Cauchy problem for the third-order (in time) evolution equation: 
\begin{align}\label{Eq-third-order}
\begin{cases}
u_{ttt}-\Delta u_{tt}+2\Delta^2 u_t-\Delta^3 u=0,&x\in\mb{R}^n,\ t>0,\\
u(0,x)=u_0(x),\ u_t(0,x)=u_1(x),\ u_{tt}(0,x)=-\Delta^2 u_0(x)-\Delta\theta_0(x),&x\in\mb{R}^n.
\end{cases}
\end{align}
\begin{remark}
We analyze the classification of \eqref{Eq-third-order} from the scalar evolution equation's viewpoint. Indeed, the partial differential operator appearing on \eqref{Eq-third-order}, namely,
\begin{align*}
	\ml{L}:=\partial_t^3-\Delta\partial_t^2+2\Delta^2\partial_t-\Delta^3,
\end{align*}
is not of Kovalevskian type (one may check \cite[Section 3.1]{Ebert-Reissig-book}). Its principal part in the sense of Petrovsky is given by the operator $\ml{L}$ itself with the symbol
\begin{align*}
	-i\tau^3-|\zeta|^2\tau^2+2i|\zeta|^4\tau+|\zeta|^6=-i\tau(\tau^2-2|\zeta|^4)-|\zeta|^2(\tau^2-|\zeta|^4).
\end{align*}
%If the linear partial differential operator $\ml{L}$ is a 2-evolution operator (see \cite[Chapter 3, Definition 3.2]{Ebert-Reissig-book}), then the corresponding equation of \eqref{PS1} has only real and distinct roots for all $\zeta\in\mb{R}^n\backslash\{0\}$. That is to say $\tau^2=2|\zeta|^4$ due to its imaginary coefficient whereas it yields a contradiction. For this reason, 
For this reason, $\ml{L}$ is not a 2-evolution operator (see \cite[Chapter 3, Definition 3.2]{Ebert-Reissig-book}).
\end{remark}
\noindent After getting some results of the vertical displacement $u=u(t,x)$, according to
\begin{align}\label{Eq-2.1}
\theta=(-\Delta)^{-1}u_{tt}-\Delta u,
\end{align}
we also can derive qualitative (even quantitative) properties of the temperature $\theta=\theta(t,x)$ immediately. Again, the symbol of $(-\Delta)^{-1}$ is $|\xi|^{-2}$.
%\begin{remark}
%Another strategy of cancellations is the use of $\theta=\theta(t,x)$ such that
%\begin{align*}
%0&=(\partial_t^2+\Delta^2)(\theta_t-\Delta\theta)-\Delta\partial_t(\partial_t^2+\Delta^2)u\\
%&=\theta_{ttt}-\Delta \theta_{tt}+\Delta^2\theta_t+\Delta \theta_t-\Delta^3\theta,
%\end{align*}
%which seems to be more mysterious than the one in \eqref{Eq-third-order} since the superfluous term $\Delta\theta_t$.
%\end{remark}
\subsection{Pointwise estimates of solutions in the Fourier space}
As preparations for deriving $L^2$ estimates of solutions to the linear Cauchy problem \eqref{Eq-third-order}, let us initially focus on the model in the Fourier space, namely,
\begin{align}\label{Eq-third-Fourier}
	\begin{cases}
		\hat{u}_{ttt}+|\xi|^2 \hat{u}_{tt}+2|\xi|^4 \hat{u}_t+|\xi|^6 \hat{u}=0,&\xi\in\mb{R}^n,\ t>0,\\
		\hat{u}(0,\xi)=\hat{u}_0(\xi),\ \hat{u}_t(0,\xi)=\hat{u}_1(\xi),\ \hat{u}_{tt}(0,\xi)=-|\xi|^4 \hat{u}_0(\xi)+|\xi|^2\hat{\theta}_0(\xi),&\xi\in\mb{R}^n,
	\end{cases}
\end{align}
where the partial Fourier transform with respect to spatial variables $x$ was employed.  Its corresponding characteristic equation
\begin{align}\label{CC-01}
	\lambda^3+|\xi|^2\lambda^2+2|\xi|^4\lambda+|\xi|^6=0
\end{align}
has one real root and two complex (non-real) conjugate roots as follows:
\begin{align*}
\lambda_1&=-a_0|\xi|^2:=-\frac{1+\alpha_-}{3}|\xi|^2,\\
	\lambda_{2,3}&=-a_1|\xi|^2\mp ia_2|\xi|^2:=-\frac{2-\alpha_-}{6}|\xi|^2\mp i\frac{\sqrt{3}\alpha_+}{6}|\xi|^2,
\end{align*}
where two positive numbers $\alpha_{\pm}$ are defined in \eqref{alpha+-}. For the readers' convenience, we state $a_0\approx 0.57$, $a_1\approx 0.22$ and $a_2\approx 1.31$. It is remarkable that $\lambda_1<0$ and $\Re \lambda_{2,3}<0$ for any $\xi\in\mb{R}^n\backslash\{0\}$. Hence, for the sake of pairwise distinct characteristic roots, the solution to \eqref{Eq-third-Fourier} is uniquely given by
\begin{align}\label{Rep-u}
\hat{u}(t,\xi)&=\widehat{K}_0(t,|\xi|)\hat{u}_0(\xi)+\widehat{K}_1(t,|\xi|)\hat{u}_1(\xi)+\widehat{K}_2(t,|\xi|)\left(-|\xi|^4\hat{u}_0(\xi)+|\xi|^2\hat{\theta}_0(\xi)\right)\notag\\
&=\left(\widehat{K}_0(t,|\xi|)-|\xi|^4\widehat{K}_2(t,|\xi|)\right)\hat{u}_0(\xi)+\widehat{K}_1(t,|\xi|)\hat{u}_1(\xi)+|\xi|^2\widehat{K}_2(t,|\xi|)\hat{\theta}_0(\xi),
\end{align} 
in which the kernels in the Fourier space have the representations
\begin{align*}
	\widehat{K}_0(t,|\xi|)&:=\sum\limits_{j=1,2,3}\frac{\exp(\lambda_jt)\prod_{k=1,2,3,\ k\neq j}\lambda_k}{\prod_{k=1,2,3,\ k\neq j}(\lambda_j-\lambda_k)},\\
	\widehat{K}_1(t,|\xi|)&:=-\sum\limits_{j=1,2,3}\frac{\exp(\lambda_jt)\sum_{k=1,2,3,\ k\neq j}\lambda_k}{\prod_{k=1,2,3,\ k\neq j}(\lambda_j-\lambda_k)},\\
	\widehat{K}_2(t,|\xi|)&:=\sum\limits_{j=1,2,3}\frac{\exp(\lambda_jt)}{\prod_{k=1,2,3,\ k\neq j}(\lambda_j-\lambda_k)}.
\end{align*}
From the Fourier transform of \eqref{Eq-2.1}, the solution $\hat{\theta}=\hat{\theta}(t,\xi)$ is represented by
\begin{align*}
	\hat{\theta}(t,\xi)&=(|\xi|^{-2}\partial_t^2+|\xi|^2)\left(\widehat{K}_0(t,|\xi|)-|\xi|^4\widehat{K}_2(t,|\xi|)\right)\hat{u}_0(\xi)\\
	&\quad+ (|\xi|^{-2}\partial_t^2+|\xi|^2)\widehat{K}_1(t,|\xi|)\hat{u}_1(\xi)+(\partial_t^2+|\xi|^4)\widehat{K}_2(t,|\xi|)\hat{\theta}_0(\xi).
\end{align*}

By lengthy but straightforward computations, one may observe
\begin{align*}
\widehat{K}_0(t,|\xi|)&=\frac{a_1^2+a_2^2}{(a_0-a_1)^2+a_2^2}\mathrm{e}^{-a_0|\xi|^2t}+\frac{a_0^2-2a_0a_1}{(a_0-a_1)^2+a_2^2}\cos(a_2|\xi|^2t)\mathrm{e}^{-a_1|\xi|^2t}\\
&\quad+\frac{a_0(a_0a_1-a_1^2+a_2^2)}{a_2[(a_0-a_1)^2+a_2^2]}\sin(a_2|\xi|^2t)\mathrm{e}^{-a_1|\xi|^2t},
\end{align*}
and with the aid of $1-\cos y=2\sin^2(y/2)$,
\begin{align}\label{K-1}
\widehat{K}_1(t,|\xi|)&=\frac{2a_1}{(a_0-a_1)^2+a_2^2}\frac{1}{|\xi|^2}\left( \mathrm{e}^{-a_0|\xi|^2t}-\cos(a_2|\xi|^2t)\mathrm{e}^{-a_1|\xi|^2t}\right)+\frac{a_0^2+a_2^2-a_1^2}{(a_0-a_1)^2+a_2^2}\frac{\sin(a_2|\xi|^2t)}{a_2|\xi|^2}\mathrm{e}^{-a_1|\xi|^2t}\notag\\
&=\frac{2a_1}{(a_0-a_1)^2+a_2^2}\frac{1}{|\xi|^2}\Big(\left(1-\cos(a_2|\xi|^2t)\right)\mathrm{e}^{-a_1|\xi|^2t}+\mathrm{e}^{-a_0|\xi|^2t}\left(1-\mathrm{e}^{(a_0-a_1)|\xi|^2t}\right)\Big)\notag\\
&\quad+\frac{a_0^2+a_2^2-a_1^2}{(a_0-a_1)^2+a_2^2}\frac{\sin(a_2|\xi|^2t)}{a_2|\xi|^2}\mathrm{e}^{-a_1|\xi|^2t}\notag\\
&=\frac{4a_1}{(a_0-a_1)^2+a_2^2}\frac{\left|\sin\left(\frac{a_2}{2}|\xi|^2t\right)\right|^2}{|\xi|^2}\mathrm{e}^{-a_1|\xi|^2t}+\frac{2a_1(a_1-a_0)}{(a_0-a_1)^2+a_2^2}t\mathrm{e}^{-a_0|\xi|^2t}\int_0^1\mathrm{e}^{(a_0-a_1)|\xi|^2t\tau}\mathrm{d}\tau\notag\\
&\quad+\frac{a_0^2+a_2^2-a_1^2}{(a_0-a_1)^2+a_2^2}\frac{\sin(a_2|\xi|^2t)}{a_2|\xi|^2}\mathrm{e}^{-a_1|\xi|^2t},
\end{align}
moreover, by the similar approach to the last chain, we claim
\begin{align*}
\widehat{K}_2(t,|\xi|)&=\frac{1}{(a_0-a_1)^2+a_2^2}\frac{1}{|\xi|^4}\left(\mathrm{e}^{-a_0|\xi|^2t}-\cos(a_2|\xi|^2t)\mathrm{e}^{-a_1|\xi|^2t}\right)+\frac{a_0-a_1}{(a_0-a_1)^2+a_2^2}\frac{\sin(a_2|\xi|^2t)}{a_2|\xi|^4}\mathrm{e}^{-a_1|\xi|^2t}\\
&=\frac{2}{(a_0-a_1)^2+a_2^2}\frac{\left|\sin\left(\frac{a_2}{2}|\xi|^2t\right)\right|^2}{|\xi|^4}\mathrm{e}^{-a_1|\xi|^2t}+\frac{a_1-a_0}{(a_0-a_1)^2+a_2^2}\frac{t}{|\xi|^2}\mathrm{e}^{-a_0|\xi|^2t}\int_0^1\mathrm{e}^{(a_0-a_1)|\xi|^2t\tau}\mathrm{d}\tau\\
&\quad+\frac{a_0-a_1}{(a_0-a_1)^2+a_2^2}\frac{\sin(a_2|\xi|^2t)}{a_2|\xi|^4}\mathrm{e}^{-a_1|\xi|^2t}.
\end{align*}
\begin{remark}\label{Rem-3.3}
We explain the reason for the previous technical treatment. Let us take $\widehat{K}_1(t,|\xi|)$ in \eqref{K-1} as an example. In the direct expression of this kernel, the first crucial term is
\begin{align*}
	I_{\mathrm{Cru}}(t,|\xi|):=\frac{1}{|\xi|^2}\left(\mathrm{e}^{-a_0|\xi|^2t}-\cos(a_2|\xi|^2t)\mathrm{e}^{-a_1|\xi|^2t}\right)
\end{align*}
with a strong singularity $|\xi|^{-2}$ for lower dimensions as $|\xi|\to0$. In other words, by a standard approach of the triangle inequality and bounded cosine function, we only can have
\begin{align*}
\|\chi_{\intt}(\xi)I_{\mathrm{Cru}}(t,|\xi|)\|_{L^2}^2&\leqslant\left\|\chi_{\intt}(\xi)\frac{1}{|\xi|^{2}}\mathrm{e}^{-a_0|\xi|^2t}\right\|_{L^2}^2+\left\|\chi_{\intt}(\xi)\frac{|\cos(a_2|\xi|^2t)|}{|\xi|^{2}}\mathrm{e}^{-a_1|\xi|^2t}\right\|_{L^2}^2\\
&\lesssim\int_0^{\varepsilon_0}r^{n-5}\mathrm{e}^{-cr^2t}\mathrm{d}r.
\end{align*}
Therefore, this rough pretreatment may lead to a strong singularity as $r\downarrow 0$ concerning lower dimensions $n\leqslant 4$. Instead, by our strategy in \eqref{K-1}, we obtain
\begin{align*}
\chi_{\intt}(\xi)\frac{\left|\sin\left(\frac{a_2}{2}|\xi|^2t\right)\right|^2}{|\xi|^2}\mathrm{e}^{-a_1|\xi|^2t}.
\end{align*}
It provides a way for us to compensate the singularity $|\xi|^{-2}$ as $|\xi|\to 0$ by  recovering the oscillating component, whose justification
will be shown in the next subsection. This is exactly one of advantages for using our reduction methodology. Our approach also can be used even if we do not get the explicit representations. The alternative idea is to employ asymptotic expansions (see, for example, \cite{Chen-Ikehata=2021,Chen-Ikehata-Palmieri=2023} in third-order (in time) PDEs of acoustic waves).
\end{remark}
Summarizing the last discussions and using
\begin{align*}
\mathrm{e}^{-a_0|\xi|^2t}\int_0^1\mathrm{e}^{(a_0-a_1)|\xi|^2t\tau}\mathrm{d}\tau\leqslant \mathrm{e}^{-a_1|\xi|^2t}
\end{align*}
since $a_0-a_1=\alpha_-/2>0$, we are able to conclude some pointwise estimates.
\begin{prop}\label{Prop-pointwise}
The following pointwise estimates for the solutions hold in the Fourier space:
\begin{align*}
\chi_{\intt}(\xi)|\hat{u}(t,\xi)|&\lesssim\chi_{\intt}(\xi)\mathrm{e}^{-c|\xi|^2t}|\hat{u}_0(\xi)|+\chi_{\intt}(\xi)\left(t+\frac{|\sin(a_2|\xi|^2t)|}{|\xi|^2}\right)\mathrm{e}^{-c|\xi|^2t}\left(|\hat{u}_1(\xi)|+|\hat{\theta}_0(\xi)|\right),\\
\big(1-\chi_{\intt}(\xi)\big)|\hat{u}(t,\xi)|&\lesssim \big(1-\chi_{\intt}(\xi)\big)\mathrm{e}^{-ct}\left(|\hat{u}_0(\xi)|+\frac{1}{|\xi|^2}\left(|\hat{u}_1(\xi)|+|\hat{\theta}_0(\xi)|\right)\right),
\end{align*}
and
\begin{align*}
	\chi_{\intt}(\xi)|\hat{\theta}(t,\xi)|&\lesssim\chi_{\intt}(\xi)|\xi|^2\mathrm{e}^{-c|\xi|^2t}|\hat{u}_0(\xi)|+\chi_{\intt}(\xi)\mathrm{e}^{-c|\xi|^2t}\left(|\hat{u}_1(\xi)|+|\hat{\theta}_0(\xi)|\right),\\
	\big(1-\chi_{\intt}(\xi)\big)|\hat{\theta}(t,\xi)|&\lesssim \big(1-\chi_{\intt}(\xi)\big)\mathrm{e}^{-ct}\left(|\xi|^2|\hat{u}_0(\xi)|+|\hat{u}_1(\xi)|+|\hat{\theta}_0(\xi)|\right),
\end{align*}
with a suitable constant $c>0$.
\end{prop}
\begin{proof}
From  the explicit formulas of kernels, one derives
\begin{align*}
	\chi_{\intt}(\xi)\left|\widehat{K}_0(t,|\xi|)-|\xi|^4\widehat{K}_2(t,|\xi|)\right|&\lesssim \chi_{\intt}(\xi)\mathrm{e}^{-c|\xi|^2t},\\
	\chi_{\intt}(\xi)|\widehat{K}_1(t,|\xi|)|+\chi_{\intt}(\xi)|\xi|^2|\widehat{K}_2(t,|\xi|)|&\lesssim \chi_{\intt}(\xi)\left(t+\frac{|\sin(a_2|\xi|^2t)|}{|\xi|^2}\right)\mathrm{e}^{-c|\xi|^2t},
\end{align*}
as well as
\begin{align*}
	\big(1-\chi_{\intt}(\xi)\big)\left|\widehat{K}_0(t,|\xi|)-|\xi|^4\widehat{K}_2(t,|\xi|)\right|&\lesssim \big(1-\chi_{\intt}(\xi)\big)\mathrm{e}^{-ct},\\
	\big(1-\chi_{\intt}(\xi)\big)|\widehat{K}_1(t,|\xi|)|+\big(1-\chi_{\intt}(\xi)\big)|\xi|^2|\widehat{K}_2(t,|\xi|)|&\lesssim \big(1-\chi_{\intt}(\xi)\big)|\xi|^{-2}\mathrm{e}^{-ct},
\end{align*}
with a suitable constant $c>0$. Then, we may arrive at the desired estimates of $\hat{u}(t,\xi)$ by plugging into the representation \eqref{Rep-u}.  For another, due to the fact that
\begin{align*}
|\partial_t^2\widehat{K}_j(t,|\xi|)|\lesssim |\xi|^{4-2j}\mathrm{e}^{-c|\xi|^2t}
\end{align*}
for $j=0,1,2$, we can employ 
\begin{align*}
	|\hat{\theta}(t,\xi)|\leqslant \frac{1}{|\xi|^2}|\hat{u}_{tt}(t,\xi)|+|\xi|^2|\hat{u}(t,\xi)|
\end{align*}
to derive the aim estimates of $\hat{\theta}(t,\xi)$. The proof is complete eventually.
\end{proof}

Next, let us introduce two multipliers in the Fourier space (they are the Fourier transforms of \eqref{M0} and \eqref{M1}, respectively) as follows:
\begin{align*}
\widehat{J}_0(t,|\xi|)&:=\frac{1}{(a_0-a_1)^2+a_2^2}\frac{1}{|\xi|^2}\left(\mathrm{e}^{-a_0|\xi|^2t}-\cos(a_2|\xi|^2t)\mathrm{e}^{-a_1|\xi|^2t}\right),\\
\widehat{J}_1(t,|\xi|)&:=\frac{1}{(a_0-a_1)^2+a_2^2}\frac{\sin(a_2|\xi|^2t)}{a_2|\xi|^2}\mathrm{e}^{-a_1|\xi|^2t},
\end{align*} 
and those generalized from the operator $|D|^{-2}\partial_t^2+|D|^2$ as follows:
\begin{align}
	\widehat{J}_{2+j}(t,|\xi|):=(|\xi|^{-2}\partial_t^2+|\xi|^2)\widehat{J}_j(t,|\xi|),
\end{align}
with $j=0,1$. Then, we can arrive at the next result easily.
\begin{prop}\label{Prop-refine}
The following refined estimates for the solutions hold in the Fourier space:
\begin{align*}
\chi_{\intt}(\xi)\left|\hat{u}(t,\xi)-\widehat{J}_0(t,|\xi|)\widehat{\Psi}_0(\xi)-\widehat{J}_1(t,|\xi|)\widehat{\Psi}_1(\xi)\right|&\lesssim\chi_{\intt}(\xi)\mathrm{e}^{-c|\xi|^2t}|\hat{u}_0(\xi)|,\\
\chi_{\intt}(\xi)\left|\hat{\theta}(t,\xi)-\widehat{J}_2(t,|\xi|)\widehat{\Psi}_0(\xi)-\widehat{J}_3(t,|\xi|)\widehat{\Psi}_1(\xi)\right|&\lesssim\chi_{\intt}(\xi)|\xi|^2\mathrm{e}^{-c|\xi|^2t}|\hat{u}_0(\xi)|,
\end{align*}
with a suitable constant $c>0$, where the combined datum in the Fourier space are defined by
\begin{align*}
	\widehat{\Psi}_0(\xi):=2a_1\hat{u}_1(\xi)+\hat{\theta}_0(\xi)\ \ \mbox{and}\ \ \widehat{\Psi}_1(\xi):=(a_0^2+a_2^2-a_1^2)\hat{u}_1(\xi)+(a_0-a_1)\hat{\theta}_0(\xi).
\end{align*}
\end{prop}
\begin{remark}\label{Rem-Strucutre-J3,4}
We realize the leading terms of $\hat{\theta}(t,\xi)$ being constituted by the linear combination of
\begin{align*}
\mathrm{e}^{-a_0|\xi|^2t},\ \ \cos(a_2|\xi|^2t)\mathrm{e}^{-a_1|\xi|^2t},\ \ \sin(a_2|\xi|^2t)\mathrm{e}^{-a_1|\xi|^2t}.
\end{align*}
Moreover, the leading terms of $\hat{\theta}(t,\xi)$ are the pseudo-differential operator with the symbol $|\xi|^{-2}\partial_t^2+|\xi|^2$ acting on those for $\hat{u}(t,\xi)$, which is strongly motivated by the first equation of \eqref{Eq-Linear-TEP}.
\end{remark}

\subsection{Optimal estimates and asymptotic profiles: Proof of Theorem \ref{Thm-Linear}}
In the last part, we found the crucial part is the combination of Fourier multipliers $\widehat{J}_0(t,|\xi|)$ and $\widehat{J}_1(t,|\xi|)$ to describe asymptotic profiles of the vertical displacement in the phase space. Our first proposition is to analyze it finely in the sense of optimal estimates as large-time.
\begin{prop}\label{Prop-J0,J1} 
	The following optimal estimates for the profile hold:
	\begin{align*}
t^{1-\frac{n}{4}}|P_{u_1}|\lesssim\left\|\chi_{\intt}(D)\big(J_0(t,|D|)P_{\Psi_0}+J_1(t,|D|)P_{\Psi_1}\big)\right\|_{L^2}\lesssim t^{1-\frac{n}{4}}\sqrt{|P_{\Psi_0}|^2+|P_{\Psi_1}|^2}
	\end{align*}
	as $t\gg1$ for any $n\geqslant 1$.
\end{prop}
\begin{proof}
	Let us denote our aim by
	\begin{align*}
		I(t):=\left\|\chi_{\intt}(D)\big(J_0(t,|D|)P_{\Psi_0}+J_1(t,|D|)P_{\Psi_1}\big)\right\|_{L^2}^2.
	\end{align*}
	We recall $\sup\limits_{z\neq0}\left|z^{-1}\sin z\right|=:L_0>0$. By the same approach as the one in \eqref{K-1}, with some positive constants $\tilde{c}$, we notice
	\begin{align}\label{Eq-new-01}
		I(t)&\lesssim\left\|\chi_{\intt}(\xi)\frac{|\sin(\tilde{c}|\xi|^2t)|}{|\xi|^2}\mathrm{e}^{-c|\xi|^2t}\right\|_{L^2}^2|P_{\Psi_0}|^2+t^2\left\|\chi_{\intt}(\xi)\mathrm{e}^{-c|\xi|^2t}\right\|_{L^2}^2|P_{\Psi_1}|^2\notag\\
		&\lesssim \int_0^{\varepsilon_0}|\sin(\tilde{c}r^2t)|^2r^{n-5}\mathrm{e}^{-2cr^2t}\mathrm{d}r|P_{\Psi_0}|^2+t^2\int_0^{\varepsilon_0}r^{n-1}\mathrm{e}^{-2cr^2t}\mathrm{d}r|P_{\Psi_1}|^2\notag\\
		&\lesssim t^{2-\frac{n}{2}}\int_0^{\varepsilon_0}\left( \frac{|\sin(\tilde{c}r^2t)|^2}{|\tilde{c}r^2t|^2}+1\right)(r^2t)^{\frac{n-1}{2}}\mathrm{e}^{-2cr^2t}\mathrm{d}(r^2t)^{\frac{1}{2}}\left(|P_{\Psi_0}|^2+|P_{\Psi_1}|^2\right)\notag\\
		&\lesssim t^{2-\frac{n}{2}}\int_0^{\varepsilon_0\sqrt{t}}\left(\frac{|\sin(\tilde{c}w^2)|^2}{|\tilde{c}w^2|^2}+1\right)w^{n-1}\mathrm{e}^{-2cw^2}\mathrm{d}w\left(|P_{\Psi_0}|^2+|P_{\Psi_1}|^2\right)\notag\\
		&\lesssim t^{2-\frac{n}{2}}\int_0^{\infty}w^{n-1}\mathrm{e}^{-2cw^2}\mathrm{d}w\left(|P_{\Psi_0}|^2+|P_{\Psi_1}|^2\right)\notag\\
		&\lesssim t^{2-\frac{n}{2}}\left(|P_{\Psi_0}|^2+|P_{\Psi_1}|^2\right)
	\end{align}
	for any $n\geqslant 1$ and $t\gg1$, where we applied polar coordinates and the change of variable $w=r\sqrt{t}$ in the last chains.
	
	Let us turn to the lower bound estimates. Again from polar coordinates, we arrive at
	\begin{align*}
		I(t)&\gtrsim\int_0^{\varepsilon_0}\left(\frac{1}{r^2}\left(\mathrm{e}^{-a_0r^2t}-\cos(a_2r^2t)\mathrm{e}^{-a_1r^2t}\right)P_{\Psi_0}+\frac{\sin(a_2r^2t)}{a_2r^2}\mathrm{e}^{-a_1r^2t}P_{\Psi_1}\right)^2r^{n-1}\mathrm{d}r\\
		&\gtrsim t^{2-\frac{n}{2}}\int_0^{\varepsilon_0\sqrt{t}}\left(\frac{1}{w^2}\left(\mathrm{e}^{-a_0w^2}-\cos(a_2w^2)\mathrm{e}^{-a_1w^2}\right)P_{\Psi_0}+\frac{\sin(a_2w^2)}{a_2w^2}\mathrm{e}^{-a_1w^2}P_{\Psi_1}\right)^2w^{n-1}\mathrm{d}w\\
		&\gtrsim t^{2-\frac{n}{2}}\int_0^{\delta_0}|\bar{I}(w)|^2w^{n-1}\mathrm{d}w
	\end{align*}
for large-time $t\gg1$ such that $\delta_0\leqslant \varepsilon_0\sqrt{t}$, where $\delta_0>0$ will be chosen later and we defined
\begin{align*}
	\bar{I}(w):=\frac{1}{w^2}\left(\mathrm{e}^{-a_0w^2}-\cos(a_2w^2)\mathrm{e}^{-a_1w^2}\right)P_{\Psi_0}+\frac{\sin(a_2w^2)}{a_2w^2}\mathrm{e}^{-a_1w^2}P_{\Psi_1}.
\end{align*}
Due to the facts that
\begin{align*}
\lim\limits_{\varepsilon_1\downarrow 0}\frac{\mathrm{e}^{-a_0\varepsilon_1^2}-\cos(a_2\varepsilon_1^2)\mathrm{e}^{-a_1\varepsilon_1^2}}{\varepsilon_1^2}=a_1-a_0\ \ \mbox{and}\ \ 
\lim\limits_{\varepsilon_1\downarrow 0}\frac{\sin(a_2\varepsilon_1^2)}{a_2\varepsilon_1^2}\mathrm{e}^{-a_1\varepsilon_1^2}=1,
\end{align*}
we can get
\begin{align*}
	\lim\limits_{\varepsilon_1\downarrow 0}|\bar{I}(\varepsilon_1)|=\left|(a_1-a_0)P_{\Psi_0}+P_{\Psi_1}\right|=\left((a_0-a_1)^2+a_2^2\right)|P_{u_1}|.
\end{align*}
For this reason, there exists $\delta_0$ with $0<w<\delta_0\ll 1$ such that
\begin{align*}
	|\bar{I}(w)|\geqslant\frac{1}{2}\left((a_0-a_1)^2+a_2^2\right)|P_{u_1}|.
\end{align*}
It leads to
\begin{align*}
	I(t)\gtrsim t^{2-\frac{n}{2}}|P_{u_1}|^2\int_0^{\delta_0}w^{n-1}\mathrm{d}w\gtrsim t^{2-\frac{n}{2}}|P_{u_1}|^2
\end{align*}
for $t\gg1$. Our proof is complete.
\end{proof}
\begin{prop}\label{Prop-J2,J3}
	The following optimal estimates for the profile hold:
	\begin{align*}
		t^{-\frac{n}{4}}|P_{\theta_0}|\lesssim\left\|\chi_{\intt}(D)\big(J_3(t,|D|)P_{\Psi_0}+J_4(t,|D|)P_{\Psi_1}\big)\right\|_{L^2}\lesssim t^{-\frac{n}{4}}\sqrt{|P_{\Psi_0}|^2+|P_{\Psi_1}|^2}
	\end{align*}
	as $t\gg1$ for any $n\geqslant 1$.
\end{prop}
\begin{proof}
	Recalling the structure of multipliers in Remark \ref{Rem-Strucutre-J3,4}, different from $\widehat{J}_{1,2}(t,|\xi|)$ in Proposition \ref{Prop-J0,J1}, there is no any singularity as $|\xi|\to0$. Then, by applying the ideas in \cite{Ikehata=2014,Ikehata-Onodera=2017}, the upper bound can be estimated easily. For the lower one,  following the same idea as the proof of Proposition \ref{Prop-J0,J1}, we obtain 
	\begin{align*}
		\left\|\chi_{\intt}(D)\big(J_3(t,|D|)P_{\Psi_0}+J_4(t,|D|)P_{\Psi_1}\big)\right\|_{L^2}^2\gtrsim t^{-\frac{n}{2}}\left|(a_0^2+a_2^2-a_1^2)P_{\Psi_0}-2a_1P_{\Psi_1}\right|^2\gtrsim t^{-\frac{n}{2}}|P_{\theta_0}|^2
	\end{align*}
	for $t\gg1$, which completes our proof.
\end{proof}

Moreover, we state another upper bound estimate, whose proof is the same as the one of Proposition \ref{Prop-J0,J1} with weaker singularities, precisely,
\begin{align*}
	\int_0^{\varepsilon_0}|\sin(r^2t)|^2r^{n-3}\mathrm{e}^{-2cr^2t}\mathrm{d}r&\lesssim t^{1-\frac{n}{2}}\int_0^{\varepsilon_0^2t}\left|\frac{\sin w}{w}\right|^2 w^{\frac{n}{2}}\mathrm{e}^{-2cw}\mathrm{d}w\lesssim t^{1-\frac{n}{2}}
\end{align*}
for large-time $t\gg1$.
\begin{coro}\label{Coro-new}
	The following upper bound estimates for the multiplier hold:
	\begin{align*}
		\left\|\chi_{\intt}(D)\frac{|\sin(|D|^2t)|}{|D|}\mathrm{e}^{-c|D|^2t}\right\|_{L^2}\lesssim t^{\frac{1}{2}-\frac{n}{4}},
	\end{align*}
	as $t\gg1$ for any $n\geqslant 1$.
\end{coro}

Motivated by the estimates \eqref{Eq-new-01}, some applications of the Plancherel theorem and the Hausdorff-Young inequality in Proposition \ref{Prop-pointwise} yield immediately that
\begin{align*}
\|u(t,\cdot)\|_{L^2}&\lesssim\|\chi_{\intt}(\xi)\hat{u}(t,\xi)\|_{L^2}+\left\|\big(1-\chi_{\intt}(\xi)\big)\hat{u}(t,\xi)\right\|_{L^2}\\
&\lesssim \left\|\chi_{\intt}(\xi)\mathrm{e}^{-c|\xi|^2t}\right\|_{L^2}\|u_0\|_{L^1}+\left\|\chi_{\intt}(\xi)\left(t+\frac{|\sin(a_2|\xi|^2t)|}{|\xi|^2}\right)\mathrm{e}^{-c|\xi|^2t}\right\|_{L^2}\left(\|u_1\|_{L^1}+\|\theta_0\|_{L^1}\right)\\
&\quad+\mathrm{e}^{-ct}\left(\|u_0\|_{L^2}+\|u_1\|_{L^2}+\|\theta_0\|_{L^2}\right)\\
&\lesssim t^{-\frac{n}{4}}\|u_0\|_{L^2\cap L^1}+t^{1-\frac{n}{4}}\left(\|u_1\|_{L^2\cap L^1}+\|\theta_0\|_{L^2\cap L^1}\right),
\end{align*}
and
\begin{align*}
	\|\theta(t,\cdot)\|_{L^2}&\lesssim\left\|\chi_{\intt}(\xi)|\xi|^2\mathrm{e}^{-c|\xi|^2t}\right\|_{L^2}\|u_0\|_{L^1}+\left\|\chi_{\intt}(\xi)\mathrm{e}^{-c|\xi|^2t}\right\|_{L^2}\left(\|u_1\|_{L^1}+\|\theta_0\|_{L^1}\right)\\
	&\quad+\mathrm{e}^{-ct}\left(\|u_0\|_{\dot{H}^2}+\|u_1\|_{L^2}+\|\theta_0\|_{L^2}\right)\\
	&\lesssim t^{-1-\frac{n}{4}}\|u_0\|_{\dot{H}^2\cap L^1}+t^{-\frac{n}{4}}\left(\|u_1\|_{L^2\cap L^1}+\|\theta_0\|_{L^2\cap L^1}\right),
\end{align*}
as $t\gg1$, where we applied Proposition \ref{Prop-J0,J1} from the upper side. By an analogous way as well as Proposition \ref{Prop-refine}, we obtain
\begin{align}\label{I2t}
	&\|u(t,\cdot)-J_0(t,|D|)\Psi_0(\cdot)-J_1(t,|D|)\Psi_1(\cdot)\|_{L^2_{\chi}}\notag\\
	&\qquad= \left\|\big(\widehat{K}_0(t,|\xi|)-|\xi|^4\widehat{K}_2(t,|\xi|)\big)\hat{u}_0(\xi)\right\|_{L^2_{\chi}}\lesssim t^{-\frac{n}{4}}\|u_0\|_{L^1},
\end{align} 
and
\begin{align*}
	&\|\theta(t,\cdot)-J_2(t,|D|)\Psi_0(\cdot)-J_3(t,|D|)\Psi_1(\cdot)\|_{L^2_{\chi}}\\
	&\qquad= \left\|(|\xi|^{-2}\partial_t^2+|\xi|^2)\big(\widehat{K}_0(t,|\xi|)-|\xi|^4\widehat{K}_2(t,|\xi|)\big)\hat{u}_0(\xi)\right\|_{L^2_{\chi}}\lesssim t^{-1-\frac{n}{4}}\|u_0\|_{ L^1},
\end{align*} 
as $t\gg1$. Employing the triangle inequality, one has
\begin{align*}
\|u(t,\cdot)-J_0(t,\cdot)P_{\Psi_0}-J_1(t,\cdot)P_{\Psi_1}\|_{L^2_{\chi}}&\leqslant \|\hat{u}(t,\xi)-\widehat{J}_0(t,|\xi|)\widehat{\Psi}_0(\xi)-\widehat{J}_1(t,|\xi|)\widehat{\Psi}_1(\xi)\|_{L^2_{\chi}}\\
&\quad+\left\|\widehat{J}_0(t,|\xi|)\big(\widehat{\Psi}_0(\xi)-P_{\Psi_0}\big)\right\|_{L^2_{\chi}}+\left\|\widehat{J}_1(t,|\xi|)\big(\widehat{\Psi}_1(\xi)-P_{\Psi_1}\big)\right\|_{L^2_{\chi}}\\
&=:I_2(t)+I_3(t)+I_4(t).
\end{align*}
Indeed, the derived estimate \eqref{I2t} implies
\begin{align*}
	I_2(t)\lesssim t^{-\frac{n}{4}}\|u_0\|_{ L^1}.
\end{align*}
Because of the decomposition of combined datum $\widehat{\Psi}_j(\xi)-P_{\Psi_j}=A_j(\xi)-iB_j(\xi)$  for $j=0,1$ with 
\begin{align*}
A_j(\xi):=\int_{\mb{R}^n}\Psi_j(x)\big(\cos(x\cdot\xi)-1\big)\mathrm{d}x\ \ \mbox{and}\ \ B_j(\xi):=\int_{\mb{R}^n}\Psi_j(x)\sin(x\cdot\xi)\mathrm{d}x.	
\end{align*}
 By using  \cite[Lemma 2.2]{Ikehata=2014}, the auxiliary functions can be estimated by
\begin{align*}
	|A_j(\xi)|+|B_j(\xi)|\lesssim|\xi|\,\|\Psi_j\|_{L^{1,1}}
\end{align*}
for $j=0,1$.  Therefore, similarly to the second line of \eqref{Eq-new-01}, it gives
\begin{align*}
	I_3(t)+I_4(t)&\lesssim\sum\limits_{j=0,1}\|\chi_{\intt}(\xi)|\xi|\widehat{J}_j(t,|\xi|)\|_{L^2}\|\Psi_j\|_{L^{1,1}}\\
	&\lesssim \left(\,\left\|\chi_{\intt}(\xi)\frac{|\sin(\tilde{c}|\xi|^2t)|}{|\xi|}\mathrm{e}^{-c|\xi|^2t}\right\|_{L^2}+t\left\|\chi_{\intt}(\xi)|\xi|\mathrm{e}^{-c|\xi|^2t}\right\|_{L^2}\right)\left(\|u_1\|_{L^{1,1}}+\|\theta_0\|_{L^{1,1}}\right)\\
	&\lesssim t^{\frac{1}{2}-\frac{n}{4}}\left(\|u_1\|_{L^{1,1}}+\|\theta_0\|_{L^{1,1}}\right)
\end{align*}
as $t\gg1$, in which we considered Corollary \ref{Coro-new}. In other words,
\begin{align}\label{Error-u}
	\|u(t,\cdot)-J_0(t,\cdot)P_{\Psi_0}-J_1(t,\cdot)P_{\Psi_1}\|_{L^2_{\chi}}\lesssim t^{-\frac{n}{4}}\|u_0\|_{ L^1}+t^{\frac{1}{2}-\frac{n}{4}}\left(\|u_1\|_{L^{1,1}}+\|\theta_0\|_{L^{1,1}}\right)
\end{align}
for large-time.  Similarly, one may also obtain
\begin{align}\label{Error-theta}
	\|\theta(t,\cdot)-J_2(t,\cdot)P_{\Psi_0}-J_3(t,\cdot)P_{\Psi_1}\|_{L^2_{\chi}}\lesssim t^{-1-\frac{n}{4}}\|u_0\|_{ L^1}+t^{-\frac{1}{2}-\frac{n}{4}}\left(\|u_1\|_{L^{1,1}}+\|\theta_0\|_{L^{1,1}}\right)
\end{align}
for large-time.  The last two estimates demonstrate Corollary \ref{Coro-Linear}. Lastly, since \eqref{Error-u}, \eqref{Error-theta} as well as Propositions \ref{Prop-J0,J1}, \ref{Prop-J2,J3} holding for large-time, Minkowski's inequality results
\begin{align*}
\|u(t,\cdot)\|_{L^2}&\geqslant\|J_0(t,\cdot)P_{\Psi_0}+J_1(t,\cdot)P_{\Psi_1}\|_{L^2_{\chi}}-\|u(t,\cdot)-J_0(t,\cdot)P_{\Psi_0}-J_1(t,\cdot)P_{\Psi_1}\|_{L^2_{\chi}}\\
&\gtrsim t^{1-\frac{n}{4}}|P_{u_1}|-t^{-\frac{n}{4}}\|u_0\|_{ L^1}-t^{\frac{1}{2}-\frac{n}{4}}\left(\|u_1\|_{L^{1,1}}+\|\theta_0\|_{L^{1,1}}\right)\gtrsim t^{1-\frac{n}{4}}|P_{u_1}|
\end{align*}
for $t\gg1$ due to $|P_{u_1}|\neq0$ and $u_0\in L^1$ as well as $u_1,\theta_0\in L^{1,1}$, moreover,
\begin{align*}
	\|\theta(t,\cdot)\|_{L^2}\gtrsim t^{-\frac{n}{4}}|P_{\theta_0}|- t^{-1-\frac{n}{4}}\|u_0\|_{ L^1}-t^{-\frac{1}{2}-\frac{n}{4}}\left(\|u_1\|_{L^{1,1}}+\|\theta_0\|_{L^{1,1}}\right)\gtrsim t^{-\frac{n}{4}}|P_{\theta_0}|.
\end{align*}
 In conclusion, our proof of Theorem \ref{Thm-Linear} is complete.

\section{Singular limits problem for the vanishing thermal parameter}\label{Sec-Singular-limit}
Let us turn to the second purpose of this paper. For the sake of readability and exactitude, we restate classical thermoelastic plate equations \eqref{Eq-Linear-TEP} by the next way:
\begin{align}\label{Eq-u,theta-epsilon}
\begin{cases}
u_{tt}^{\epsilon}+\Delta^2 u^{\epsilon}+\Delta\theta^{\epsilon}=0,&x\in\mb{R}^n,\ t>0,\\
\epsilon\theta_t^{\epsilon}-\Delta\theta^{\epsilon}-\Delta u_t^{\epsilon}=0,&x\in\mb{R}^n,\ t>0,\\
u^{\epsilon}(0,x)=u_0(x),\ u_t^{\epsilon}(0,x)=u_1(x),\ \theta^{\epsilon}(0,x)=\theta_0(x),&x\in\mb{R}^n,
\end{cases}
\end{align}
where $\epsilon>0$ denotes the thermal parameter. We recall that the formal limit model with $\epsilon=0$ was introduced by \eqref{Eq-u,theta,0}. Thus, our aim is to demonstrate $u^{\epsilon}\to u^0$ in some norms, and find the second-order profile for the small $\epsilon$. Before rigorously demonstrating the desired convergence results, it is important to understand the influence of the small thermal parameter $\epsilon$, which will provide some profiles in a formal way. Throughout this section, we set $\epsilon>0$ to be a small number, and $C>0$ to be a suitable constant independent of $t$ as well as $\epsilon$.
\subsection{Formal analysis of thermoelastic plate equations}
In order to analyze the influence of small parameter and asymptotic profiles of solution to \eqref{Eq-u,theta-epsilon} with respect to $0<\epsilon\ll 1$, motivated by the boundary layer theory and the multi-scale analysis (see, for example, \cite{Holmes=1995}), the solution $(u^{\epsilon},\theta^{\epsilon})$ of thermoelastic plate equations \eqref{Eq-u,theta-epsilon} owns the expansions
\begin{align}\label{S4-repre}
\begin{cases}
\displaystyle{u^{\epsilon}(t,x)=\sum\limits_{j\geqslant0}\epsilon^j\left(u^{I,j}(t,x)+u^{L,j}(\epsilon^{-1}t,x)\right),}\\
\displaystyle{\theta^{\epsilon}(t,x)=\sum\limits_{j\geqslant0}\epsilon^j\left(\theta^{I,j}(t,x)+\theta^{L,j}(\epsilon^{-1}t,x)\right),}
\end{cases}
\end{align}
where all terms in the last representations are assumed to be sufficiently smooth. Here, $u^{I,j}=u^{I,j}(t,x)$, $\theta^{I,j}=\theta^{I,j}(t,x)$ stand for the dominant profiles for each order expansion, and $u^{L,j}=u^{L,j}(z,x)$, $\theta^{L,j}=\theta^{L,j}(z,x)$ with $z:=\epsilon^{-1}t$ denote the profiles decaying to zero as $z\to\infty$. These components will be fixed later. 

By employing the perturbation theory, we plug the new ansatz \eqref{S4-repre} into two equations of \eqref{Eq-u,theta-epsilon}. We immediately arrive at
\begin{align*}
	\begin{cases}
	\displaystyle{\sum\limits_{j\geqslant 0}\epsilon^j(u^{I,j}_{tt}+\epsilon^{-2}u^{L,j}_{zz})+\sum\limits_{j\geqslant0}\epsilon^j(\Delta^2 u^{I,j}+\Delta^{2}u^{L,j})+\sum\limits_{j\geqslant0}\epsilon^j(\Delta \theta^{I,j}+\Delta \theta^{L,j})=0,}\\
	\displaystyle{\sum\limits_{j\geqslant0}\epsilon^j(\epsilon\theta^{I,j}_t+\theta^{L,j}_z)-\sum\limits_{j\geqslant0}\epsilon^j(\Delta\theta^{I,j}+\Delta\theta^{L,j})-\sum\limits_{j\geqslant0}\epsilon^j(\Delta u^{I,j}_t+\epsilon^{-1}\Delta u^{L,j}_z)=0.  }
	\end{cases}
\end{align*}
Matching the order of $\epsilon^{j+1}$ (here, we consider $j\in\mb{Z}$ formally), we claim 
\begin{align}\label{S4-02}
\begin{cases}
u^{I,j+1}_{tt}+\Delta^2u^{I,j+1}+\Delta \theta^{I,j+1}=-u^{L,j+3}_{zz}-\Delta^2 u^{L,j+1}-\Delta\theta^{L,j+1},\\
\theta^{I,j}_t-\Delta\theta^{I,j+1}-\Delta u^{I,j+1}_t=-\theta^{L,j+1}_z+\Delta\theta^{L,j+1}+\Delta u_z^{L,j+2}.
\end{cases}
\end{align}
Letting $\epsilon\downarrow0$, i.e. $z\to\infty$, the right-hand sides of \eqref{S4-02} tend to zero so that
\begin{align}\label{S4-03}
\begin{cases}
u^{I,j+1}_{tt}+\Delta^2 u^{I,j+1}+\Delta\theta^{I,j+1}=0,\\
\theta^{I,j}_t-\Delta\theta^{I,j+1}-\Delta u^{I,j+1}_t=0,
\end{cases}
\end{align}
with $u^{I,j}\equiv0\equiv \theta^{I,j}$ if $j$ becomes negative, as well as
\begin{align}\label{S4-04}
	\begin{cases}
	u^{L,j+3}_{zz}+\Delta^2u^{L,j+1}+\Delta\theta^{L,j+1}=0,\\
	\theta^{L,j+1}_z-\Delta\theta^{L,j+1}-\Delta u^{L,j+2}_z=0,
	\end{cases}
\end{align}
with $u^{L,j}\equiv0\equiv \theta^{L,j}$ if $j$ becomes negative. Later we will derive the profiles by carefully  determining their initial conditions.

From the initial conditions in the original problem \eqref{Eq-u,theta-epsilon} with the ansatz \eqref{S4-repre}, one notices
\begin{align}\label{S4-N-01}
\begin{cases}
\displaystyle{u_0(x)=u^{I,0}(0,x)+u^{L,0}(0,x)+\sum\limits_{j\geqslant 1}\epsilon^j\left(u^{I,j}(0,x)+u^{L,j}(0,x)\right),}\\
\displaystyle{u_1(x)=\epsilon^{-1}u^{L,0}_z(0,x)+u^{I,0}_t(0,x)+u^{L,1}_z(0,x)+\sum\limits_{j\geqslant 1}\epsilon^j\left(u^{I,j}_t(0,x)+u^{L,j+1}_z(0,x)\right),}\\
\displaystyle{\theta_0(x)=\theta^{I,0}(0,x)+\theta^{L,0}(0,x)+\sum\limits_{j\geqslant 1}\epsilon^j\left(\theta^{I,j}(0,x)+\theta^{L,j}(0,x)\right).}
\end{cases}
\end{align}
Under our consideration that three initial datum in \eqref{Eq-u,theta-epsilon} are independent of $\epsilon$ in general, we may naturally take vanishing value of $u^{I,j}(0,x)$, $u^{L,j}(0,x)$, $u^{I,j}_t(0,x)$, $u^{L,j+1}_z(0,x)$ as well as $\theta^{I,j}(0,x)$, $\theta^{L,j}(0,x)$ for any $j\geqslant 1$.  The condition \eqref{S4-N-01} can be reset as
\begin{align*}
	\begin{cases}
	u^{I,0}(0,x)+u^{L,0}(0,x)=u_0(x),\\
	u^{I,0}_t(0,x)+u^{L,1}_z(0,x)=u_1(x),\ u^{L,0}_z(0,x)=0,\\
	\theta^{I,0}(0,x)+\theta^{L,0}(0,x)=\theta_0(x).
	\end{cases}
\end{align*}
Indeed, we did not fix all initial datum in the present step. In the next subsection, we will follow the methodology of derivations for the Prandtl equation in the boundary layer theory to get the explicit formula of initial conditions.

\subsection{Formal derivation of asymptotic profiles for the small thermal parameter}
To begin with our deductions, let us take $j=-1$ in \eqref{S4-03}, and $j=-3,-2$ in \eqref{S4-04}, respectively, to find
\begin{align}\label{S4-05}
\begin{cases}
u^{I,0}_{tt}+\Delta^2u^{I,0}+\Delta\theta^{I,0}=0,&x\in\mb{R}^n,\ t>0,\\
-\Delta\theta^{I,0}-\Delta u_t^{I,0}=0,&x\in\mb{R}^n,\ t>0,\\
u^{I,0}(0,x)=u_0^{I,0}(x),\ u^{I,0}_t(0,x)=u^{I,0}_1(x),&x\in\mb{R}^n,
\end{cases}
\end{align}
in which \eqref{S4-05}$_1+$\eqref{S4-05}$_2$ yields the plate equation with structural damping $u^{I,0}_{tt}+\Delta^2u^{I,0}-\Delta u^{I,0}_t=0$, moreover,
\begin{align}\label{S4-06}
\begin{cases}
u^{L,0}_{zz}=0,\ \ u^{L,1}_{zz}=0,\ \ \Delta u^{L,0}_z=0,&x\in\mb{R}^n,\ z>0,\\
u^{L,0}(0,x)=u_0^{L,0}(x),\ u^{L,0}_z(0,x)=u^{L,1}(0,x)=0,\ u^{L,1}_z(0,x)=u_1^{L,1}(x), &x\in\mb{R}^n,
\end{cases}
\end{align}
with $u_0^{I,0}(x)+u^{L,0}_0(x)=u_0(x)$ and $u^{I,0}_1(x)+u^{L,1}_1(x)=u_1(x)$. By constructing new functions $u^{P,j}=u^{P,j}(t,x)$ with $j=0,1$ such that   
\begin{align}\label{S4-07}
u^{P,0}(t,x):=u^{L,0}(\epsilon^{-1}t,x)+u^{I,0}_0(x),\ \ u^{P,1}(t,x):=\epsilon u^{L,1}(\epsilon^{-1}t,x)+tu^{I,0}_1(x),
\end{align}
they satisfy the following equations (we used \eqref{S4-06} actually):
\begin{align*}
\begin{cases}
u^{P,0}_{tt}=0,\ \ u_{tt}^{P,1}=0,&x\in\mb{R}^n,\ t>0,\\
u^{P,0}(0,x)=u_0(x),\ u^{P,0}_t(0,x)=u^{P,1}(0,x)=0,\ u^{P,1}_t(0,x)=u_1(x),&x\in\mb{R}^n.
\end{cases}
\end{align*}
The solutions of the previous models can be easily obtained by $u^{P,0}(t,x)=u_0(x)$ and $u^{P,1}(t,x)=tu_1(x)$. Again with $\epsilon\downarrow0$ in \eqref{S4-07}, it gives $u^{I,0}_0(x)=u_0(x)$ as well as $u^{I,0}_1(x)=u_1(x)$, respectively. Coming back to \eqref{S4-06}, due to all vanishing initial datum, we arrive at $u^{L,0}(z,x)\equiv0\equiv u^{L,1}(z,x)$ for any $z\geqslant0$ and $x\in\mb{R}^n$. Furthermore, we claim $u^{I,0}(t,x)\equiv u^0(t,x)$ which is the solution to the structurally damped plate equation \eqref{Eq-u,theta,0}. The formal expansion \eqref{S4-repre} so far can be updated by
\begin{align*}
\begin{cases}
\displaystyle{u^{\epsilon}(t,x)=u^0(t,x)+\sum\limits_{j\geqslant 1}\epsilon^ju^{I,j}(t,x)+\sum\limits_{j\geqslant 2}\epsilon^j u^{L,j}(\epsilon^{-1}t,x)},\\
\displaystyle{\theta^{\epsilon}(t,x)=\theta^{I,0}(t,x)+\theta^{L,0}(t,x)+\sum\limits_{j\geqslant 1}\left(\theta^{I,j}(t,x)+\theta^{L,j}(\epsilon^{-1}t,x)\right)}.
\end{cases}
\end{align*}

To determine higher-order profiles of $u^{\epsilon}$, we may choose $j=0$ in \eqref{S4-03} to see
\begin{align}\label{S4-08}
\begin{cases}
u^{I,1}_{tt}+\Delta^2 u^{I,1}-\Delta u^{I,1}_t=-\theta_t^{I,0},&x\in\mb{R}^n,\ t>0,\\
u^{I,1}(0,x)=u^{I,1}_t(0,x)=0,\ \theta^{I,0}(0,x)=\theta_0^{I,0}(x),&x\in\mb{R}^n.
\end{cases}
\end{align}
 The second relation in \eqref{S4-05} may imply $\theta^{I,0}(t,x)=-u^{I,0}_t(t,x)$, therefore, $\theta^{I,0}(0,x)=:\theta^{I,0}_0(x)=-u_1(x)$. In addition, taking $j=-1$ in \eqref{S4-04}, we are able to find 
\begin{align*}
\begin{cases}
u^{L,2}_{zz}=-\Delta\theta^{L,0},&x\in\mb{R}^n,\ z>0,\\
\theta^{L,0}_z-\Delta\theta^{L,0}=0,&x\in\mb{R}^n,\ z>0,\\
u^{L,2}(0,x)=u^{L,2}_z(0,x)=0,\ \theta^{L,0}(0,x)=u_1(x)+\theta_0(x),&x\in\mb{R}^n,
\end{cases}
\end{align*}
where we employed $\theta_0^{L,0}(0,x)=\theta_0(x)-\theta^{I,0}_0(x)=u_1(x)+\theta_0(x)$.  The second equation is the standard heat model which can be uniquely solved by
\begin{align}\label{S4-heat}
\theta^{L,0}(\epsilon^{-1}t,x)&=\mathrm{e}^{-|D|^2\frac{t}{\epsilon}}\big(u_1(x)+\theta_0(x)\big)=\frac{\epsilon^{n/2}}{(4\pi t)^{n/2}}\int_{\mb{R}^n}\mathrm{e}^{-\frac{\epsilon|x-y|^2}{4t}}\big(u_1(y)+\theta_0(y)\big)\mathrm{d}y.
\end{align}
Combining the last representation and the derived differential equation, we obtain
\begin{align*}
u^{L,2}(\epsilon^{-1}t,x)=\left(\frac{1}{|D|^2}\left(\mathrm{e}^{-|D|^2\frac{t}{\epsilon}}-1\right)+\frac{t}{\epsilon}\right)\big(u_1(x)+\theta_0(x)\big).
\end{align*}
Summarizing the last statements, we are able to conclude the next result for formal higher-order profiles of the solution to \eqref{Eq-u,theta-epsilon}.
\begin{prop}\label{Prop-expansion}
The solution $(u^{\epsilon},\theta^{\epsilon})$ to the Cauchy problem for thermoelastic plate equations \eqref{Eq-u,theta-epsilon} with the small $\epsilon>0$ formally have the following asymptotic expansions:
\begin{align*}
\begin{cases}
\displaystyle{u^{\epsilon}(t,x)=u^0(t,x)+\epsilon u^{I,1}(t,x)+\sum\limits_{j\geqslant2}\epsilon^j\left(u^{I,j}(t,x)+u^{L,j}(\epsilon^{-1}t,x)\right),}\\
\displaystyle{\theta^{\epsilon}(t,x)=-u_t^0(t,x)+\mathrm{e}^{-|D|^2\frac{t}{\epsilon}}\big(u_1(x)+\theta_0(x)\big)+\sum\limits_{j\geqslant 1}\epsilon^j\left(\theta^{I,j}(t,x)+\theta^{L,j}(\epsilon^{-1}t,x)\right),}
\end{cases}
\end{align*}
in which
\begin{itemize}
	\item $u^0=u^0(t,x)$ is the solution to the structurally damped plate equation \eqref{Eq-u,theta,0};
	\item $u^{I,1}=u^{I,1}(t,x)$ is the solution to the inhomogeneous structurally damped plate equation
	\begin{align*}
	\begin{cases}
	u^{I,1}_{tt}+\Delta^2u^{I,1}-\Delta u^{I,1}_t=-\Delta^2 u^0+\Delta u^0_t,&x\in\mb{R}^n,\ t>0,\\
	u^{I,1}(0,x)=u_t^{I,1}(0,x)=0,&x\in\mb{R}^n;
	\end{cases}
	\end{align*}
\item the pair of $u^{I,j+1}=u^{I,j+1}(t,x)$ and $\theta^{I,j}=\theta^{I,j}(t,x)$ for $j\geqslant 1$ is the solution to \eqref{S4-03} with vanishing datum;
\item the pair of $u^{L,j+1}=u^{L,j+1}(z,x)$ and $\theta^{L,j}=\theta^{L,j}(z,x)$ with $z=\epsilon^{-1}t$ for $j\geqslant 1$ is the solution to \eqref{S4-04} with vanishing datum.
\end{itemize}
\end{prop}
\begin{remark}
For the reason of the Fourier law of heat conduction, according to Proposition \ref{Prop-expansion}, we did not observe any initial layer (i.e. a rapid change with respect to the small thermal parameter $\epsilon$). The profile \eqref{S4-heat} originates from the heat equation only. In other words, we expect some changes polynomially in terms of $1/\epsilon$ rather than an exponential change $\mathrm{e}^{-1/\epsilon}$ (see, for example, the initial layer between the damped waves and the heat equation \cite{Ikehata-Sobajima=2022})
\end{remark}

\subsection{Rigorous justification of singular limits: Proof of Theorem \ref{Thm-singu-lim}}
Let us introduce two quantities to  describe the first-order errors as follows:
\begin{align*}
	\begin{cases}
	U(t,x):=u^{\epsilon}(t,x)-u^0(t,x),\\
	\Theta(t,x):=\theta^{\epsilon}(t,x)-\theta^0(t,x)=\theta^{\epsilon}(t,x)+u^0_t(t,x),
	\end{cases}
\end{align*}
which fulfill the following inhomogeneous thermoelastic plate equations:
\begin{align*}
\begin{cases}
U_{tt}+\Delta^2 U+\Delta \Theta=0,&x\in\mb{R}^n,\ t>0,\\
\epsilon\Theta_t-\Delta\Theta-\Delta U_t=\epsilon F,&x\in\mb{R}^n,\ t>0,\\
U(0,x)=U_t(0,x)=0,\ \Theta(0,x)=u_1(x)+\theta_0(x),&x\in\mb{R}^n,
\end{cases}
\end{align*}
where the source term $F=F(t,x)$ can be expressed by
\begin{align*}
	F(t,x)=-\theta^0_t(t,x)=u^0_{tt}(t,x)=\Delta u^0_t(t,x)-\Delta^2 u^0(t,x)
\end{align*}
that is independent of $\epsilon$.

To prove our desired results, we follow the idea in Section \ref{Sec-Large-time} to get the first-order (in time) coupled system in the Fourier space
\begin{align*}
\begin{cases}
\widehat{U}_t-\widehat{V}=0,&\xi\in\mb{R}^n,\ t>0,\\
\widehat{V}_t-\widehat{W}=0,&\xi\in\mb{R}^n,\ t>0,\\
\epsilon\widehat{W}_t+|\xi|^2\widehat{W}+(1+\epsilon)|\xi|^4\widehat{V}+|\xi|^6\widehat{U}=\epsilon|\xi|^2\widehat{F},&\xi\in\mb{R}^n,\ t>0,\\
\widehat{U}(0,\xi)=\widehat{V}(0,\xi)=0,\ \widehat{W}(0,\xi)=|\xi|^2\big(\hat{u}_1(\xi)+\hat{\theta}_0(\xi)\big),&\xi\in\mb{R}^n,
\end{cases}
\end{align*}
where we used $\widehat{V}:=\widehat{U}_t$, $\widehat{W}:=\widehat{V}_t$ and the equation
\begin{align*}
	\epsilon\widehat{U}_{ttt}+|\xi|^2\widehat{U}_{tt}+(1+\epsilon)|\xi|^4\widehat{U}_t+|\xi|^6\widehat{U}=\epsilon|\xi|^2\widehat{F}.
\end{align*}
Then, we construct a suitable energy
\begin{align*}
\widehat{E}:=\frac{1}{2}\left(\left|\frac{1}{2}|\xi|^2\widehat{V}+\epsilon\widehat{W}\right|^2 +\frac{\epsilon}{1+\epsilon}|\xi|^4\left||\xi|^2\widehat{U}+(1+\epsilon)\widehat{V}\right|^2+\frac{1}{4}|\xi|^4|\widehat{V}|^2+\frac{1-\epsilon}{2(1+\epsilon)}|\xi|^8|\widehat{U}|^2  \right).
\end{align*} 
Taking time-derivative of the energy, we can derive
\begin{align*}
\frac{\mathrm{d}}{\mathrm{d}t}\widehat{E}&=\Re\left[\left(\frac{1}{2}|\xi|^2\widehat{V}_t+\epsilon\widehat{W}_t\right)\left(\frac{1}{2}|\xi|^2\overline{\widehat{V}}+\epsilon\overline{\widehat{W}}\right) \right]+\frac{\epsilon}{1+\epsilon}|\xi|^4\Re\left[(|\xi|^2\widehat{V}+(1+\epsilon)\widehat{W})(|\xi|^2\overline{\widehat{U}}+(1+\epsilon)\overline{\widehat{V}})\right]\\
&\quad+\frac{1}{4}|\xi|^4\Re(\widehat{W}\overline{\widehat{V}})+\frac{1-\epsilon}{2(1+\epsilon)}|\xi|^8\Re(\widehat{V}\overline{\widehat{U}})\\
&=-\frac{1-\epsilon}{2}|\xi|^6|\widehat{V}|^2-\frac{\epsilon}{2}|\xi|^2|\widehat{W}|^2+\epsilon|\xi|^2\Re\left[\widehat{F}\left(\frac{1}{2}|\xi|^2\overline{\widehat{V}}+\epsilon\overline{\widehat{W}}\right)\right]\\
&\leqslant \frac{1+4\epsilon(1-\epsilon)}{8(1-\epsilon)}\epsilon^2|\xi|^2|\widehat{F}|^2\leqslant C\epsilon^2|\xi|^2|\widehat{F}|^2.
\end{align*}
Recalling that $\widehat{F}=-|\xi|^2\hat{u}_t^0-|\xi|^4\hat{u}^0$, we need to get some estimates of $\hat{u}^0_t$ as well as $\hat{u}^0$ in the Fourier space. We apply the partial Fourier transform to \eqref{Eq-u,theta,0}, namely,
\begin{align*}
\begin{cases}
\hat{u}^0_{tt}+|\xi|^4\hat{u}^0+|\xi|^2\hat{u}_t^0=0,&\xi\in\mb{R}^n,\ t>0,\\
\hat{u}^0(0,\xi)=\hat{u}_0(\xi),\ \hat{u}^0_t(0,\xi)=\hat{u}_1(\xi),&\xi\in\mb{R}^n,
\end{cases}
\end{align*}
whose solution and its time-derivative are given by 
\begin{align*}
\hat{u}^0(t,\xi)&=\left[\frac{1}{\sqrt{3}}\sin\left(\frac{\sqrt{3}}{2}|\xi|^2t\right)+\cos\left(\frac{\sqrt{3}}{2}|\xi|^2t\right)\right]\mathrm{e}^{-\frac{1}{2}|\xi|^2t}\hat{u}_0(\xi)+\frac{2}{\sqrt{3}}\frac{\sin\left(\frac{\sqrt{3}}{2}|\xi|^2t\right)}{|\xi|^2}\mathrm{e}^{-\frac{1}{2}|\xi|^2t}\hat{u}_1(\xi),\\
\hat{u}_t^0(t,\xi)&=-\frac{2}{\sqrt{3}}|\xi|^2\sin \left(\frac{\sqrt{3}}{2}|\xi|^2t\right)\mathrm{e}^{-\frac{1}{2}|\xi|^2t}\hat{u}_0(\xi)+\left[\cos\left(\frac{\sqrt{3}}{2}|\xi|^2t\right)-\frac{1}{\sqrt{3}}\sin \left(\frac{\sqrt{3}}{2}|\xi|^2t\right)\right]\mathrm{e}^{-\frac{1}{2}|\xi|^2t}\hat{u}_1(\xi).
\end{align*}
According to 
\begin{align*}
	|\widehat{U}_t(t,\xi)|^2+|\xi|^4|\widehat{U}(t,\xi)|^2&\leqslant \frac{C}{|\xi|^4}\widehat{E}(t,\xi)\leqslant C\epsilon^2|\hat{u}_1(\xi)+\hat{\theta}_0(\xi)|^2+C\epsilon^2|\xi|^{-2}\int_0^t|\widehat{F}(\tau,\xi)|^2\mathrm{d}\tau\\
	&\leqslant C\epsilon^2|\hat{u}_1(\xi)+\hat{\theta}_0(\xi)|^2+C\epsilon^2\int_0^t\mathrm{e}^{-c|\xi|^2\tau}\mathrm{d}\tau\left(|\xi|^{6}|\hat{u}_0(\xi)|^2+|\xi|^2|\hat{u}_1(\xi)|^2\right),
\end{align*}
we integrate it over $\mb{R}^n$ to deduce
\begin{align*}
\|U_t(t,\cdot)\|_{L^2}^2+\|\Delta U(t,\cdot)\|_{L^2}^2&\leqslant C\epsilon^2\|u_1+\theta_0\|_{L^2}^2+C\epsilon^2\int_0^t\left\|\chi_{\intt}(\xi)|\xi|^3\mathrm{e}^{-c|\xi|^2\tau}\hat{u}_0(\xi)\right\|_{L^2}^2\mathrm{d}\tau\\
&\quad+C\epsilon^2\int_0^t\left\|\chi_{\intt}(\xi)|\xi|\mathrm{e}^{-c|\xi|^2\tau}\hat{u}_1(\xi)\right\|_{L^2}^2\mathrm{d}\tau+C\epsilon^2\left(\|u_0\|_{H^2}^2+\|u_1\|_{L^2}^2\right)\\
&\leqslant C\epsilon^2\|u_1+\theta_0\|_{L^2}^2+C\epsilon^2\left(\|u_0\|_{H^2}^2+\|u_1\|_{L^2}^2\right)\\
&\quad+C\epsilon^2\left(\int_0^t(1+\tau)^{-3-\frac{n}{2}}\mathrm{d}\tau\|u_0\|_{L^1}^2+\int_0^t(1+\tau)^{-1-\frac{n}{2}}\mathrm{d}\tau\|u_1\|_{L^1}^2\right)\\
&\leqslant C\epsilon^2\left(\|u_1+\theta_0\|_{L^2}^2+\|u_0\|_{H^2\cap L^1}^2+\|u_1\|_{L^2\cap L^1}^2\right),
\end{align*}
where the next estimate was used:
\begin{align*}
\int_0^t\big(1-\chi_{\intt}(\xi)\big)\mathrm{e}^{-c|\xi|^2\tau}\mathrm{d}\tau=\frac{1-\chi_{\intt}(\xi)}{c|\xi|^2}\left(1-\mathrm{e}^{-c|\xi|^2t}\right)\lesssim\frac{1-\chi_{\intt}(\xi)}{|\xi|^2}.
\end{align*}
Then, we complete the first desired consequence in the sense of derivatives.

For another, we already know
\begin{align*}
|\widehat{U}(t,\xi)|^2\leqslant C\epsilon^2\frac{|\hat{u}_1(\xi)+\hat{\theta}_0(\xi)|^2}{|\xi|^4}+C\epsilon^2\int_0^t\mathrm{e}^{-c|\xi|^2\tau}\mathrm{d}\tau\left(|\xi|^2|\hat{u}_0(\xi)|^2+\frac{|\hat{u}_1(\xi)|^2}{|\xi|^2}\right).
\end{align*}
In order to achieve global (in time) convergence result for $n\geqslant 3$, we apply $|P_{u_1}|=0$ so that
\begin{align*}
|\hat{u}_1(\xi)|^2\leqslant C\left(|P_{u_1}|^2+|\xi|^2\|u_1\|_{L^{1,1}}^2\right)=C|\xi|^2\|u_1\|_{L^{1,1}}^2.
\end{align*} 
Thus, one gets
\begin{align*}
	\|U(t,\cdot)\|_{L^2}^2&\leqslant C\epsilon^2\|u_1+\theta_0\|_{\dot{H}^{-2}}^2+C\epsilon^2\left(\|u_0\|_{L^2}^2+\|u_1\|_{L^2}^2\right)\\
	&\quad+C\epsilon^2\left(\int_0^t(1+\tau)^{-1-\frac{n}{2}}\mathrm{d}\tau\|u_0\|_{L^1}^2+\int_0^t(1+\tau)^{-\frac{n}{2}}\mathrm{d}\tau\|u_1\|_{L^{1,1}}^2\right)\\
	&\leqslant C\epsilon^2\left(\|u_1+\theta_0\|_{\dot{H}^{-2}}^2+\|u_0\|_{L^2\cap L^1}^2+\|u_1\|_{L^2\cap L^{1,1}}^2\right)
\end{align*}
for $n\geqslant 3$ since $(1+\tau)^{-\frac{n}{2}}\in L^1$ if $n\geqslant 3$. It finishes our proof.

\subsection{Rigorous justification of second-order profiles: Proof of Theorem \ref{Thm-Second}}
Strongly motivated by Proposition \ref{Prop-expansion}, we may construct the second-order error terms
\begin{align}\label{Ansatz-second}
	\begin{cases}
U^s(t,x)=u^{\epsilon}(t,x)-u^0(t,x)-\epsilon u^{I,1}(t,x),\\
\Theta^s(t,x)=\theta^{\epsilon}(t,x)-\theta^0(t,x)-\epsilon\theta^{I,1}(t,x),	
	\end{cases}
\end{align}
where the profile $u^{I,1}=u^{I,1}(t,x)$ fulfills \eqref{S4-08} and the other profile $\theta^{I,1}=\theta^{I,1}(t,x)$ actually satisfies
\begin{align*}
	\theta^{I,1}=(-\Delta)^{-1}\left(u^{I,1}_{tt}+\Delta^2u^{I,1}\right).
\end{align*}
Since $(u^{I,1},\theta^{I,1})$ is the solution to
\begin{align*}
\begin{cases}
u^{I,1}_{tt}+\Delta^2 u^{I,1}+\Delta \theta^{I,1}=0,&x\in\mb{R}^n,\ t>0,\\
\theta_t^0-\Delta \theta^{I,1}-\Delta u^{I,1}_t=0,&x\in\mb{R}^n,\ t>0,\\
u^{I,1}(0,x)=u^{I,1}_t(0,x)=\theta^{I,1}(0,x)=0,&x\in\mb{R}^n,
\end{cases}
\end{align*}
 the difference $(U^s,\Theta^s)$ solves inhomogeneous thermoelastic plate equations as follows:
\begin{align*}
\begin{cases}
	U^s_{tt}+\Delta^2 U^s+\Delta \Theta^s=0,&x\in\mb{R}^n,\ t>0,\\
	\epsilon\Theta_t^s-\Delta\Theta^s-\Delta U^s_t=\epsilon^2 F^s,&x\in\mb{R}^n,\ t>0,\\
	U^s(0,x)=U_t^s(0,x)=\Theta^s(0,x)=0,&x\in\mb{R}^n,
\end{cases}
\end{align*}
where the assumption $u_1(x)+\theta_0(x)\equiv0$ was proposed, and the source term on the second equation in the above is
\begin{align*}
F^s(t,x)&=-\theta_t^{I,1}(t,x)=\Delta^{-1}\partial_t\big(u_{tt}^{I,1}(t,x)+\Delta^2u^{I,1}(t,x)\big)\\
&=u^{I,1}_{tt}(t,x)-\Delta u_t^0(t,x)+u^0_{tt}(t,x)\\
&=-\Delta^2 u^{I,1}(t,x)+\Delta u_t^{I,1}(t,x)-\Delta^2 u^0(t,x)+\Delta u^0_t(t,x).
\end{align*}
With the aid of the same procedure as those in the last subsection, we may estimate
\begin{align*}
&|\widehat{U}_t^s(t,\xi)|^2+|\xi|^4|\widehat{U}^s(t,\xi)|^2\leqslant C\epsilon^4|\xi|^{-2}\int_0^t|\widehat{F}^s(\tau,\xi)|^2\mathrm{d}\tau\\
&\qquad\leqslant C\epsilon^4|\xi|^6\int_0^t\left(|\hat{u}^{I,1}(\tau,\xi)|^2+|\hat{u}^0(\tau,\xi)|^2\right)\mathrm{d}\tau+C\epsilon^4|\xi|^2\int_0^t\left(|\hat{u}_t^{I,1}(\tau,\xi)|^2+|\hat{u}^0_t(\tau,\xi)|^2\right)\mathrm{d}\tau.
\end{align*}
Moreover, the Fourier transform of \eqref{S4-08} is given by
\begin{align*}
\begin{cases}
\hat{u}_{tt}^{I,1}+|\xi|^2\hat{u}^{I,1}_t+|\xi|^4\hat{u}^{I,1}=-|\xi|^4\hat{u}^0-|\xi|^2\hat{u}_t^0,&\xi\in\mb{R}^n,\ t>0,\\
\hat{u}^{I,1}(0,\xi)=\hat{u}_t^{I,1}(0,\xi)=0,&\xi\in\mb{R}^n.
\end{cases}
\end{align*}
Its solution can be determined by the use of Duhamel's principle, i.e.
\begin{align*}
\hat{u}^{I,1}(t,\xi)=-\frac{2}{\sqrt{3}}\int_0^t\sin\left(\frac{\sqrt{3}}{2}|\xi|^2(t-\tau)\right)\mathrm{e}^{-\frac{1}{2}|\xi|^2(t-\tau)}\left(|\xi|^2\hat{u}^0(\tau,\xi)+\hat{u}_t^0(\tau,\xi)\right)\mathrm{d}\tau.
\end{align*}
Some direct estimates imply
\begin{align*}
|\hat{u}^{I,1}(t,\xi)|&\leqslant Ct\mathrm{e}^{-\frac{1}{2}|\xi|^2t}\left(|\xi|^2|\hat{u}_0(\xi)|+|\hat{u}_1(\xi)|\right),\\
|\hat{u}_t^{I,1}(t,\xi)|&\leqslant Ct|\xi|^2\mathrm{e}^{-\frac{1}{2}|\xi|^2t}\left(|\xi|^2|\hat{u}_0(\xi)|+|\hat{u}_1(\xi)|\right).
\end{align*}
Summarizing the last estimates, we claim
\begin{align*}
	|\widehat{U}^s_t(t,\xi)|^2+|\xi|^4|\widehat{U}^s(t,\xi)|^2&\leqslant  C\epsilon^4\int_0^t\mathrm{e}^{-c|\xi|^2\tau}\mathrm{d}\tau\left(|\xi|^{6}|\hat{u}_0(\xi)|^2+|\xi|^2|\hat{u}_1(\xi)|^2\right),\\
	|\widehat{U}^s(t,\xi)|^2&\leqslant  C\epsilon^4\int_0^t\mathrm{e}^{-c|\xi|^2\tau}\mathrm{d}\tau\left(|\xi|^2|\hat{u}_0(\xi)|^2+\frac{|\hat{u}_1(\xi)|^2}{|\xi|^2}\right).
\end{align*}
By using the same approach as those for Theorem \ref{Thm-singu-lim}, the next estimates can be done:
\begin{align*}
\|U^s_t(t,\cdot)\|_{L^2}^2+\|\Delta U^s(t,\cdot)\|_{L^2}^2\leqslant C\epsilon^4\left(\|u_0\|_{H^2\cap L^1}^2+\|u_1\|_{L^2\cap L^1}^2\right)
\end{align*}
for $n\geqslant 1$, and
\begin{align*}
	\|U^s(t,\cdot)\|_{L^2}^2&\leqslant C\epsilon^4\left(\|u_0\|_{L^2\cap L^1}^2+\|u_1\|_{L^2\cap L^{1,1}}^2\right)
\end{align*} 
for $n\geqslant 3$. Then, our proof is completed.

\section{Concluding remarks}\label{Sec-Concluding}
In the final section, we will give some remarks concerning wide applications of the reduction methodology proposed in this work. The typical example of applications is the system of thermoelasticity, which can describe the elasticity and thermal behavior of elastic, heat conducting media. Let us concentrate on the following Cauchy problem for the thermoelastic system in 1D:
\begin{align}\label{Eq-TE}
	\begin{cases}
	u_{tt}-\alpha u_{xx}+\gamma_1\theta_x=0,&x\in\mb{R},\ t>0,\\
	\theta_t-\kappa\theta_{xx}+\gamma_2u_{tx}=0,&x\in\mb{R},\ t>0,\\
	u(0,x)=u_0(x),\ u_t(0,x)=u_1(x),\ \theta(0,x)=\theta_0(x),&x\in\mb{R},
	\end{cases}
\end{align}
where $u=u(t,x)$ and $\theta=\theta(t,x)$ denote the elastic displacement and the temperature difference to the equilibrium, individually. In the model \eqref{Eq-TE}, physical constants $\alpha>0$ is the elasticity modules, $\kappa>0$ is the thermal conductivity and the thermoelastic coupling coefficients $\gamma_1,\gamma_2$ fulfill $\gamma_1\gamma_2>0$.

Following the same approach in this work, we may reduce \eqref{Eq-TE} to the third-order (in time) PDEs
\begin{align*}
	\begin{cases}
	u_{ttt}-\kappa u_{ttxx}-(\gamma_1\gamma_2+\alpha)u_{txx}+\kappa\alpha u_{xxxx}=0,&x\in\mb{R},\ t>0,\\
	u(0,x)=u_0(x),\ u_t(0,x)=u_1(x),\ u_{tt}(0,x)=\alpha u_0''(x)-\gamma_1\theta_0'(x),&x\in\mb{R}.
	\end{cases}
\end{align*}
We conjecture that by using Fourier analysis, asymptotic analysis as well as WKB method, some qualitative properties of solutions to \eqref{Eq-TE}, particularly, the solution itself $u=u(t,x)$, can be obtained. Nevertheless, it is not trivial to derive optimal estimates and optimal leading terms in the higher-dimensions $n\geqslant 2$ because of some new challenges from the coupling effect and singularities.

\section*{Acknowledgments}
 The second author was supported in part by Grant-in-Aid for scientific Research (C) 20K03682 of JSPS. The authors thank Michael Reissig (TU Bergakademie Freiberg) for the suggestions in the preparation of the paper.

% ------------------------------------------------------------------------
\end{document}